\documentclass[12pt,a4paper]{article}

\usepackage{latexsym}
\usepackage{amsmath}
\usepackage{amssymb}
\usepackage{arydshln}

\usepackage{cite}
\usepackage{stmaryrd}
\usepackage{enumerate}

\usepackage{hyperref}

\usepackage{color}
\usepackage{lineno}
\usepackage{graphicx}
\usepackage{ae}
\usepackage{amsmath}
\usepackage{amssymb}
\usepackage{latexsym}
\usepackage{url}
\usepackage{epsfig}
\usepackage{mathrsfs}
\usepackage{amsfonts}
\usepackage{amsthm}

\usepackage{float}
\usepackage{subfig}
\usepackage{tikz}
\usetikzlibrary{positioning}

\newtheorem{theorem}{Theorem}[section]

\newtheorem{lemma}[theorem]{Lemma}

\newtheorem{cor}[theorem]{Corollary}

\newtheorem{example}{Example}

\theoremstyle{definition}
\newtheorem{definition}{Definition}

\numberwithin{equation}{section} 
\allowdisplaybreaks    

\def\qed{\hfill$\Box$\vspace{12pt}}

\long\def\delete#1{}

\usepackage{xcolor}
\usepackage[normalem]{ulem}

\begin{document}
\title {Laplacian pair state transfer in Q-graph}

\author{Ming Jiang$^{a,b}$,~Xiaogang Liu$^{a,b,c,}$\thanks{Supported by the National Natural Science Foundation of China (No. 12371358) and the Guangdong Basic and Applied
Basic Research Foundation (No. 2023A1515010986).}~$^,$\thanks{ Corresponding author. Email addresses: mjiang@mail.nwpu.edu.cn, xiaogliu@nwpu.edu.cn,
wj66@mail.nwpu.edu.cn}~,~Jing Wang$^{a,b}$
\\[2mm]
{\small $^a$School of Mathematics and Statistics,}\\[-0.8ex]
{\small Northwestern Polytechnical University, Xi'an, Shaanxi 710072, P.R.~China}\\
{\small $^b$Research \& Development Institute of Northwestern Polytechnical University in Shenzhen,}\\[-0.8ex]
{\small Shenzhen, Guandong 518063, P.R. China}\\
{\small $^c$Xi'an-Budapest Joint Research Center for Combinatorics,}\\[-0.8ex]
{\small Northwestern Polytechnical University, Xi'an, Shaanxi 710129, P.R. China}\\
}
\date{}

\openup 0.5\jot
\maketitle

\begin{abstract}
In 2018, Chen and Godsil proposed the concept of Laplacian perfect pair state transfer which is a brilliant generalization of Laplacian perfect state transfer. In this
paper, we study the existence of Laplacian perfect pair state transfer in the Q-graph of an $r$-regular graph for $r\ge2$. We prove that the Q-graph of an $r$-regular graph
does not have Laplacian perfect pair state transfer when $r+1$ is prime or a power of $2$. We also give a sufficient condition for Q-graph to have Laplacian pretty good pair state transfer.

\smallskip


\end{abstract}

\section{Introduction}

Recently, quantum state transfer in graphs has got a lot of attention, since the study of quantum state transfer in graphs will provide a theoretical basis for the
construction of quantum communication networks with quantum state transfer, which also has a potential role in quantum information processing and quantum computing
\cite{BOSE2}.

Let $G$ be a simple undirected graph with the adjacency matrix $A_G$ and the degree matrix $D_G$. The Laplacian matrix of $G$ is denoted that $L_{G}=D_G -A_G $. The
\emph{transition matrix} of $G$ relative to $L_{G}$ is defined by
$$
U_{L_{G}}(t) = \exp(-\mathrm{i}tL_{G})=\sum_{k=0}^{\infty}\frac{(-\mathrm{i})^{k}L_{G}^{k}t^{k}}{k!}, ~ t \in \mathbb{R},~\mathrm{i}=\sqrt{-1}.
$$
Suppose that $u$ and $v$ are a pair of vertices of $G$. We use $\mathbf{e}_u,\mathbf{e}_v \in \mathbb{R}^{n}$ to denote the characteristic vectors of $u$ and $v$,
respectively. Let $\mathbf{e}_{u}^{\top}$ denote the transpose of $\mathbf{e}_{u}$. Based on the so-called \emph{XYZ}-model, $|\mathbf{e}_{u}^{\top}U_{L_G}(t)\mathbf{e}_{v}|$
represents the probability of the amount of information carried by a quantum state transfer in $G$ from the vertex $u$ to the vertex $v$ at time $t$ \cite{BOSE1}. If there
exists a time $\tau$  such that
$$
|\mathbf{e}_{u}^{\top}U_{L_G}(\tau)\mathbf{e}_{v}|=1,
$$
then $G$ is said to have \emph{Laplacian perfect state transfer} (short for LPST) between $u$ and $v$ at time $\tau$. Up until now, many results about LPST have been given.
For example, see \cite{Ack, Alvi, BOSE1, Coutinho11, Kirk, LiY1, LiY2, WJ}.


For a pair of vertices $a$ and $b$ of a graph $G$, we use $\mathbf e_{a}-\mathbf e_{b}$ to represent the \emph{pair state} of $\left \{ a,b \right \}$. If $a$ is adjacent to
$b$ in $G$, then $\mathbf e_{a}-\mathbf e_{b}$ is called \emph{edge state}. For two distinct pair states $\mathbf{e}_{a}-\mathbf{e}_{b}$ and $\mathbf{e}_{c}-\mathbf{e}_{d}$ of
a graph $G$, if there exists a time $\tau$ and a complex scalar $\chi $ with $\left|\chi \right|=1$ satisfying that
\begin{align*}
U_{L_{G}}(\tau)(\mathbf{e}_{a}-\mathbf{e}_{b})=\chi (\mathbf{e}_{c}-\mathbf{e}_{d}),
\end{align*}
which is equivalent to
\begin{align*}
\left|\frac{1}{2}(\mathbf e_{a}-\mathbf e_{b})^\top e^{-\mathrm{i}\tau L_{G}}(\mathbf e_{c}-\mathbf e_{d})\right|^{2}=1,
\end{align*}
then we say $G$ has \emph{Laplacian perfect pair state transfer} (short for Pair-LPST) between $\mathbf{e}_{a}-\mathbf{e}_{b}$ and $\mathbf{e}_{c}-\mathbf{e}_{d}$ at time
$\tau $. This concept was first proposed by Chen and Godsil in 2018 \cite{QCh18}. It has been show in \cite{ChG19} that for a certain number of vertices, there are more graphs
having Pair-LPST than LPST, which means that Pair-LPST is more flexible than LPST. But for all that, there are also a lot of graphs do not have Pair-LPST. In \cite{liu2023},
Wang, Liu and Wang proposed the definition of \emph{Laplacian pretty good pair state transfer} (short for Pair-LPGST), which is  a relaxation of Pair-LPST. We say a graph $G$
has Pair-LPGST between $\mathbf e_{a}-\mathbf e_{b}$ and $\mathbf e_{c}-\mathbf e_{d}$ if for any $\epsilon>0$, there exists a time $\tau$ such that
\begin{align*}
\left|\frac{1}{2}(\mathbf e_{a}-\mathbf e_{b})^\top e^{-\mathrm{i}\tau L_{G}}(\mathbf e_{c}-\mathbf e_{d})\right|^{2}>1-\epsilon.
\end{align*}
Note that if $\mathbf e_{a}-\mathbf e_{b}$ and $\mathbf e_{c}-\mathbf e_{d}$ are edge states, then Pair-LPST and Pair-LPGST can be called Edge-LPST and Edge-LPGST,
respectively.

The study of Pair-LPST and Pair-LPGST is a very new topic and there are only few results now. In 2020, Chen and Godsil gave some useful and important results on Pair-LPST
\cite {ChG19}. In 2021, Cao gave a characterization of cubelike graphs having Edge-LPST and they gave methods to obtain some classes of infinite graphs having Edge-LPST \cite
{CAO}. In 2022, Cao and Wan gave a characterization of abelian Cayley graphs having Edge-LPST \cite {CAO2}. In 2023, Wang, Liu and Wang gave sufficient conditions for vertex
corona to have or not have Pair-LPST and they also gave a sufficient condition for vertex corona to have Pair-LPGST  \cite {liu2023}.

In this paper, we investigate the existence of Pair-LPST and Pair-LPGST in the Q-graph of a regular graph.

\begin{definition}\label{Q} (Cvetkovi{\'c} et al.\cite{Q-gr})
The \emph{Q-graph} of a graph $G$, denoted by $\mathcal{Q} \left ( G \right )$, is the graph obtained from $G$ by inserting a new vertex into each edge of $G$, and joining by
edges those pairs of new vertices which lie on adjacent edges of $G$. We denote the set of such new vertices by $I(G)$.
\end{definition}

The paper is organized as follows. In Section \ref{Sec:pre}, we list some useful results. In Section \ref{Sec:Q},  we study the existence of Pair-LPST in the Q-graph of an
$r$-regular graph for $r\ge2$. We prove that the Q-graph of an $r$-regular graph does not have Pair-LPST when $r+1$ is prime or a power of $2$. In Section
\ref{Sec:LPST3-3}, we give a
sufficient condition for Q-graph to have Pair-LPGST.


\section{Preliminaries}\label{Sec:pre}


For a graph $G$ with Laplacian matrix $L_G$. Let $\theta_0> \theta_1>\cdots>\theta_p=0 $ be all distinct Laplacian eigenvalues with multiplicity $l_r$, $r=0,1,\ldots,p$. Let
$\left\{\mathbf{x}_{1}^{(r)}, \ldots, \mathbf{x}_{l_r}^{(r)}\right\}$ denote the orthonormal basis of the eigenspace of $\theta_{r}$. The set of all distinct Laplacian eigenvalues is denoted by $\mathrm{{Spec}}_{L}(G)$. We use $\mathbf{x}^H$ to denote the conjugate transpose of a column vector $\mathbf{x}$. For each $\theta_{r}$ of
$\mathrm{{Spec}}_{L}(G)$, we define
$$
F_{\theta_r} = \sum\limits_{i=1}^{l_r}\mathbf{x}_i^{(r)} \left(\mathbf{x}_i^{(r)}\right)^H,
$$
which is called the eigenprojector corresponding of $\theta_{r}$. Note that $\sum\limits_{r=0}^pF_{\theta_r}=I$, where $I$ is the identity matrix. Then
\begin{equation}
\label{spect1}
L_G=\sum_{r=0}^{p}\theta_r F_{\theta_r},
\end{equation}
which is called the \emph{spectral decomposition} of $L_G$ with respect to the distinct eigenvalues. Note that $F_{\theta_r}^{2}=F_{\theta_r}$ and
$F_{\theta_r}F_{\theta_s}=\mathbf{0}$ for $r\neq s$, where $\mathbf{0}$ denotes the zero matrix. By $(\ref{spect1})$, we have
\begin{equation}
\label{spect2}
U_{L_G}(t)=\sum_{r=0}^{p}\exp(-it\theta _{r})F_{{\theta_r}}.
\end{equation}
The \emph{Laplacian eigenvalue support} of pair state $\mathbf{e}_{a}-\mathbf{e}_{b}$ of $G$ is defined that
$$
\mathrm{{supp}}_{L_G}(\mathbf{e}_{a}-\mathbf{e}_{b})=\left \{ \theta \mid  \theta \in \mathrm{{Spec}}_{L}(G),\ F_\theta\mathbf(\mathbf{e}_{a}-\mathbf{e}_{b})\ne \mathbf{0}
\right \}.
$$
Let $\mathbf{e}_{a}-\mathbf{e}_{b}$ and $\mathbf{e}_{c}-\mathbf{e}_{d}$ be two pair states of $G$. If
$$
F_\theta\mathbf(\mathbf{e}_{a}-\mathbf{e}_{b})=\pm F_\theta\mathbf(\mathbf{e}_{c}-\mathbf{e}_{d}),
$$
for each eigenvalue $\theta$ of $L_G$, then we say $\mathbf{e}_{a}-\mathbf{e}_{b}$ and $\mathbf{e}_{c}-\mathbf{e}_{d}$ are \emph{Laplacian strongly cospectral}. Denote that
$$
\Lambda^{+}_{ab,cd}=\left \{ \theta \mid  \theta \in \mathrm{{supp}}_{L_G}(\mathbf{e}_{a}-\mathbf{e}_{b}),\ F_\theta\mathbf(\mathbf{e}_{a}-\mathbf{e}_{b})=F_\theta\mathbf(\mathbf{e}_c-\mathbf{e}_d) \right
\},
$$
and
$$
\Lambda^{-}_{ab,cd}=\left \{ \theta \mid  \theta \in \mathrm{{supp}}_{L_G}(\mathbf{e}_{a}-\mathbf{e}_{b}),\ F_\theta\mathbf(\mathbf{e}_{a}-\mathbf{e}_{b})=-F_\theta\mathbf(\mathbf{e}_c-\mathbf{e}_d) \right
\}.
$$

The following result gives a necessary and sufficient condition for a graph to have Pair-LPST.

\begin{lemma}\emph{(See \cite[Theorem~3.9]{ChG19})}\label{Coutinho}
Let $\mathbf{e}_a-\mathbf{e}_b$ and $\mathbf{e}_c-\mathbf{e}_d$ be two distinct pair states in a graph $G$. Set
$S=\mathrm{supp}_{L_G}(\mathbf{e}_{a}-\mathbf{e}_{b})=\left \{ \theta_{r} \mid  0\le r \le k \right\}$ with $\theta_{0}\in \Lambda^{+}_{ab,cd}$. Then $G$ has Pair-LPST between $\mathbf{e}_a-\mathbf{e}_b$ and $\mathbf{e}_c-\mathbf{e}_d$ in $G$ if and only if all of
the following hold:
\begin{itemize}
\item[\rm (a)]  $\mathbf{e}_a-\mathbf{e}_b$ and $\mathbf{e}_c-\mathbf{e}_d$ are Laplacian strongly cospectral;
 \item[\rm (b)] The eigenvalues in $S$ are either all integers or all quadratic integers. Moreover, there is a square-free integer $\Delta$ such that each $\theta \in S$ is
    a quadratic integer in $\mathbb{Q}(\sqrt{\Delta})$, and the difference of any two eigenvalues in $S$ is an integer multiple of $\sqrt{\Delta}$. Here, we allow $\Delta=1$ for the case where all eigenvalues in $S$ are integers;
\item[\rm (c)] Let $g=\gcd\left(\left\{\frac{\theta_{0}-\theta_{r}}{\sqrt{\Delta}}\right\}^{k}_{r=0}\right)$. Then
\begin{itemize}
\item[\rm(i)] $\theta_{r}\in\Lambda^{+}_{ab,cd}$ if and only if $\frac{\theta_{0}-\theta_{r}}{g\sqrt{\Delta}}$ is even, and
 \item[\rm(ii)] $\theta_{r}\in\Lambda^{-}_{ab,cd}$ if and only if $\frac{\theta_{0}-\theta_{r}}{g\sqrt{\Delta}}$ is odd.
\end{itemize}
\end{itemize}

 If these conditions hold and Pair-LPST occurs between $\mathbf{e}_a-\mathbf{e}_b$ and $\mathbf{e}_c-\mathbf{e}_d$ at time $\tau$, then the minimum time is
 $\tau_0=\frac{\pi}{g\sqrt{\Delta}}$.
\end{lemma}




Recall Lemma \ref{Coutinho} (c) that if $G$ has Pair-LPST between two distinct pair states $\mathbf{e}_a-\mathbf{e}_b$ and $\mathbf{e}_c-\mathbf{e}_d$, then eigenvalues in $S$
can be divided into $\Lambda^{+}_{ab,cd}$ and $\Lambda^{-}_{ab,cd}$. We wonder if one of $\Lambda^{\pm}_{ab,cd}$ can be empty. If one of $\Lambda^{\pm}_{ab,cd}$ is empty, say
$\Lambda^{-}_{ab,cd}$ is empty, then we have
$$
\mathbf{e}_a-\mathbf{e}_b=\sum_{r=0}^pF_{\theta_r}(\mathbf{e}_a-\mathbf{e}_b)=\sum_{\theta_r\in S}F_{\theta_r}(\mathbf{e}_a-\mathbf{e}_b)=\sum_{\theta_r\in
S}F_{\theta_r}(\mathbf{e}_c-\mathbf{e}_d)=\mathbf{e}_c-\mathbf{e}_d,
$$
a contradiction to the fact that $\mathbf{e}_a-\mathbf{e}_b$ and $\mathbf{e}_c-\mathbf{e}_d$ are two distinct pair states. Similarly, we can also get a contradiction if
$\Lambda^{+}_{ab,cd}$ is empty. Thus, neither of $\Lambda^{\pm}_{ab,cd}$ can be empty, which implies that $\left |S  \right | \ge2$ . Then we get the following lemma which is
a generalization of \cite[Theorem~5.1.4]{QCh18}.


\begin{lemma}\label{support}
Let $\mathbf e_{a}-\mathbf e_{b}$ be a pair state of $G$ such that the eigenvalue support of $\mathbf e_{a}-\mathbf e_{b}$ has size one. Then $\mathbf e_{a}-\mathbf e_{b}$
does not have Pair-LPST in $G$.
\end{lemma}

To study the existence of Pair-LPGST in $\mathcal{Q} \left ( G \right )$, we need the Kronecker's Approximation Theorem stated as follows.
\begin{lemma}
\emph{(See \cite[Theorem~442]{Hw})}
\label{H-W}
Let $1,\lambda_1,\ldots,\lambda_m$ be linearly independent over $\mathbb{Q}$. Let $\alpha_1,\ldots,\alpha_m$ be arbitrary real numbers and $\epsilon$ be a positive real
number. Then there exist integers $l$ and  $q_1,\ldots, q_m$ such that
\begin{equation}
\label{KroApp}
|l\lambda_k-\alpha_k-q_k|<\epsilon,
\end{equation}
for $k=1,\ldots,m$.
\end{lemma}

For convenience, (\ref{KroApp}) will be represented as $l\lambda_k-q_k\approx\alpha_k$.

\begin{lemma}
\emph{(See \cite[Theorem~1~(a)]{Ri}) }
\label{Ri}
Let $p_1,p_2,\ldots,p_k$ be distinct positive  primes.
The set $$\left\{\sqrt[n]{p_1^{m(1)}\cdots p_k^{m(k)}}: 0\leq m(i)<n,~1\leq i \leq k \right\}$$ is linearly independent over  $\mathbb{Q}$.
\end{lemma}

The following result comes from Lemma \ref{Ri} immediately.

\begin{cor}\label{independent}
\label{Ri1}
The set $$\left\{\sqrt{\Delta}: \Delta\text{~is~a~square-free~integer}\right\}$$ is linearly independent over $\mathbb{Q}$.
\end{cor}

The following results give the Laplacian eigenvalues and eigenprojectors of Q-graph. We use $R_{G}$ to denote the incidence matrix of $G$. Let $J_{n \times m}$ be an $n \times
m$ matrix with each entry equals $1$. For a vector $\mathbf{y}$, we write $\left \| \mathbf{y} \right \|$ as the module of $\mathbf{y}$.

\begin{lemma}\emph{(See \cite[Theorem~3.1]{LiY2}}\label{eigenvalues1}
Suppose that $G$ is an $r$-regular non-bipartite graph with $n$ vertices, $m$ edges and $r\ge 2$. Suppose that $\left \{ \mathbf{y}_{1} ,\mathbf{y}_{2},\ldots,\mathbf{y}_{m-n}
\right \} $ is a maximal independent solution vector set of the linear system $R_{G}\mathbf{y}=\mathbf{0}$. Then
\begin{itemize}
\item[\rm (a)]  $2r+2$ is a Laplacian eigenvalue of $\mathcal{Q} \left ( G \right ) $, and its corresponding eigenprojector is
\begin{align}\label{F41}
     F_{2r+2}=\sum_{l=1}^{m-n}\frac{1}{\left \| \mathbf{y}_l \right \|^2 }
 \left(
	\begin{array}{cc}
	\mathbf{0}&\mathbf{0} \\ [0.3cm]
    \mathbf{0}&  \mathbf{y}_l\mathbf{y}_l^\top
	\end{array}
	\right).
	\end{align}
\item[\rm (b)]  For each Laplacian eigenvalue $\theta $ of $G$,
 \begin{align}\label{F441}
\theta^{\pm }=\frac{1}{2} \left ( \theta +2+r\pm\sqrt{\left ( r+2-\theta \right )^2 +4\theta}  \right ),
 \end{align}
are Laplacian eigenvalues of $\mathcal{Q} \left ( G \right )$, and their  corresponding eigenvectors are
              \begin{align}\label{F42}
     F_{\theta^{\pm}} =&\frac{\left (\theta+2-\theta^{\pm}  \right )^2 }{\left (\theta+2-\theta^{\pm}  \right )^2+\left (2r-\theta \right )}
 \left(
	\begin{array}{cc}
	F_{\theta}(G)&\frac{1}{\theta+2-\theta^{\pm}} F_{\theta}(G)R_{G}\\ [0.3cm]
	\frac{1}{\theta+2-\theta^{\pm}} \left (  F_{\theta}(G)R_{G} \right )^{\top }  &    \frac{1}{\left (\theta+2-\theta^{\pm}\right )^{2}}R_{G}^{\top }F_{\theta}(G)R_{G}
	\end{array}
	\right),
	\end{align}
   where $F_{\theta}(G)$     denotes the eigenprojector corresponding to $\theta$ of $G$.
    \end{itemize}
Thus, the spectral decomposition of $L_{\mathcal{Q} \left ( G \right )}$   is written as
\begin{align}\label{I1}
   L_{\mathcal{Q} \left ( G \right )}=\sum_{\theta \in \mathrm{{Spec}}_{L}(G)}\sum_{\pm }\theta^{\pm }F_{\theta^\pm}+\left ( 2r+2 \right )F_{2r+2}.
\end{align}
\end{lemma}

\begin{lemma}\emph{(See \cite[Theorem~3.2]{LiY2}}\label{eigenvalues}
Suppose that $G$ is an $r$-regular bipartite graph with $n$ vertices, $m$ edges and $r\ge 2$. Suppose that $\left \{ \mathbf{y}_{1} ,\mathbf{y}_{2},\ldots,\mathbf{y}_{m-n+1}
\right \} $ is a maximal set of independent solution vectors of the linear system $R_{G}\mathbf{y}=\mathbf{0}$. Then
\begin{itemize}
    \item[\rm (a)]  $2r+2$ is a Laplacian eigenvalue of $\mathcal{Q} \left ( G \right ) $, and its corresponding eigenprojector is
\begin{align}\label{F43}
     F_{2r+2}=\sum_{l=1}^{m-n+1}\frac{1}{\left \| \mathbf{y}_l \right \|^2 }
 \left(
	\begin{array}{cc}
	\mathbf{0}&\mathbf{0}\\ [0.3cm]
   \mathbf{0}&  \mathbf{y}_l\mathbf{y}_l^\top
	\end{array}
	\right).
	\end{align}

        \item[\rm (b)]  For each Laplacian eigenvalue $\theta \neq2r$ of $G$,
         \begin{align}\label{F441--1}
\theta^{\pm }=\frac{1}{2} \left ( \theta +2+r\pm\sqrt{\left ( r+2-\theta \right )^2 +4\theta}  \right ),
 \end{align}
                are Laplacian eigenvalues of $\mathcal{Q} \left ( G \right )$, and their  corresponding eigenvectors are
              \begin{align}\label{F44}
     F_{\theta^{\pm}} =&\frac{\left (\theta+2-\theta^{\pm}  \right )^2 }{\left (\theta+2-\theta^{\pm}  \right )^2+\left (2r-\theta \right )}
 \left(
	\begin{array}{cc}
	F_{\theta}(G)&\frac{1}{\theta+2-\theta^{\pm}} F_{\theta}(G)R_{G}\\ [0.3cm]
	\frac{1}{\theta+2-\theta^{\pm}} \left (  F_{\theta}(G)R_{G} \right )^{\top }  &    \frac{1}{\left (\theta+2-\theta^{\pm}\right )^{2}}R_{G}^{\top }F_{\theta}(G)R_{G}
	\end{array}
	\right),
	\end{align}
   where $F_{\theta}(G)$     denotes the eigenprojector corresponding to $\theta$ of $G$.
       \item[\rm (c)] $r$ is a Laplacian eigenvalue  of $\mathcal{Q} \left ( G \right ) $, and its corresponding eigenprojector is
              \begin{align}\label{F45}
     F_{r} =\frac{1}{n}
\begin{pmatrix}
 J_{\left | V_1 \right |} & -J_{\left | V_1 \right |\times\left | V_2 \right |} & \mathbf{0}\\
 - J_{\left | V_2 \right |\times\left | V_1 \right |} & J_{\left | V_2 \right |} & \mathbf{0}\\
 \mathbf{0} & \mathbf{0} &\mathbf{0}
\end{pmatrix},
	\end{align}
where $V_{1}\cup V_{2}$ is the bipartition of $V(G).$
 \end{itemize}
Thus, the spectral decomposition of $L_{\mathcal{Q} \left ( G \right )}$   is written as
\begin{align}\label{I2}
   L_{\mathcal{Q} \left ( G \right )}=\sum_{\theta \in \mathrm{{Spec}}_{L}(G)\setminus \left \{ 2r \right \} }\sum_{\pm }\theta ^{\pm }F_{\theta^ \pm}+\left ( 2r+2 \right
   )F_{2r+2}+rF_{r}.
\end{align}
\end{lemma}


\section{Pair-LPST in Q-graph} \label{Sec:Q}

 Let $\mathbf{e}_{i}^{n}$ denote an $n$-dimensional unit vector with the $i$-th entry
being 1. According to Definition \ref{Q}, the vertex set of Q-graph is denoted by
$$
V (\mathcal{Q} \left ( G \right ) )=V(G)\cup I(G).
$$
Then there are three types of $\mathbf{e}_{a}^{m+n}-\mathbf{e}_{b}^{m+n}$ in Q-graph:
   $$
   a,b \in I(G);~~ a \in V(G),  b \in I(G); ~~a,b \in V(G).
   $$
In this section, we will give some sufficient conditions for these three types of pair states of Q-graph not having Pair-LPST.

 \begin{theorem}\label{C_n}
 Let $G$ be an r-regular graph with $n$ vertices, $m$ edges and $r\ge 2$.
  If
 $$
 a,b\in V(G) \text{~or~} a,b \in I(G),
 $$
  then for each $ \theta \in \mathrm{{Spec}}_{L}(G)\setminus \left \{ 2r \right \}$,
  $$
  \theta^{+}\in \mathrm{{supp}}_{L_{\mathcal{Q} \left ( G \right )}}(\mathbf{e}_{a}^{m+n}-\mathbf{e}_{b}^{m+n})
  \text{~if~and~only~if~}
  \theta^{-}\in \mathrm{{supp}}_{L_{\mathcal{Q} \left ( G \right )}}(\mathbf{e}_{a}^{m+n}-\mathbf{e}_{b}^{m+n}).
  $$
Moreover, if $\mathbf{e}_{a}^{m+n}-\mathbf{e}_{b}^{m+n}$ and $ \mathbf{e}_{c}^{m+n}-\mathbf{e}_{d}^{m+n}$ are two pair states of $\mathcal{Q} \left ( G \right )$ that are
Laplacian strongly cospectral,  then
$$
\theta^{+} \in \Lambda^{+}_{ab,cd}\text{~if~and~only~if~}\theta^{-} \in \Lambda^{+}_{ab,cd},
$$
and
$$
\theta^{+} \in \Lambda^{-}_{ab,cd} \text{~if~and~only~if~} \theta^{-} \in \Lambda^{-}_{ab,cd}.
$$
\end{theorem}
\begin{proof}\noindent\emph{Case 1.} $a , b \in V(G)$. By (\ref{F42}) and (\ref{F44}), we have
\begin{align}\label{Q1}
 	F_{\theta^{\pm}}\left ( \mathbf{e}_{a}^{m+n}-\mathbf{e}_{b}^{m+n}\right )=\frac{\left (\theta+2-\theta^{\pm}  \right )^2 }{\left (\theta+2-\theta^{\pm}  \right )^2+\left
(2r-\theta \right )}\begin{pmatrix}
F_{\theta}(G)(\mathbf{e}_{a}^{n}-\mathbf{e}_{b}^{n}) \\[0.2cm] \frac{1}{\theta+2-\theta^{\pm}} \left (  F_{\theta}(G)R_{G} \right )^{\top
}(\mathbf{e}_{a}^{n}-\mathbf{e}_{b}^{n})
\end{pmatrix}.
	\end{align}
Obviously,
\begin{equation}\label{PMandG--1}
F_{\theta^{\pm}}\left ( \mathbf{e}_{a}^{m+n}-\mathbf{e}_{b}^{m+n}\right )\neq\mathbf{0} \text{~if~and~only~if~}
F_\theta(G) \left ( \mathbf{e}_{a}^{n}-\mathbf{e}_{b}^{n}\right )\neq\mathbf{0},
\end{equation}
which implies that $\theta^{+}\in \mathrm{{supp}}_{L_{\mathcal{Q} \left ( G \right )}}(\mathbf{e}_{a}^{m+n}-\mathbf{e}_{b}^{m+n})$ if and only if
  $\theta^{-}\in \mathrm{{supp}}_{L_{\mathcal{Q} \left ( G \right )}}(\mathbf{e}_{a}^{m+n}-\mathbf{e}_{b}^{m+n})$.

Moreover, if $\mathbf{e}_{a}^{m+n}-\mathbf{e}_{b}^{m+n}$ and $ \mathbf{e}_{c}^{m+n}-\mathbf{e}_{d}^{m+n}$ are Laplacian strongly cospectral in  $\mathcal{Q} \left ( G \right
)$, then for each $\theta^{\pm}\in \mathrm{{Spec}}_{L}(\mathcal{Q} \left ( G \right ))$, we have
\begin{equation}\label{Eigen73}
F_{\theta^{\pm}}\left ( \mathbf{e}_{a}^{m+n}-\mathbf{e}_{b}^{m+n}\right )=\pm F_{\theta^{\pm}}\left ( \mathbf{e}_{c}^{m+n}-\mathbf{e}_{d}^{m+n}\right ).
\end{equation}
By (\ref{F42}) and (\ref{F44}),  we have
\begin{equation}\label{Eigen72}
	F_{0^+}=F_{r+2}=\frac{r}{n(r+2)}
  \left(	
  \begin{array}{cc}
	J_{n\times n}& \frac{-2}{r}J_{n\times m} \\ [0.2cm]
	 \frac{-2}{r}J_{m\times n} & \frac{4}{r^2}J_{m\times m}
	\end{array}
	\right).
	\end{equation}
Thus, by (\ref{Eigen73}) and (\ref{Eigen72}),
 $$
  F_{r+2} (\mathbf{e}_{a}^{n+m}-\mathbf{e}_{b}^{n+m})=\pm F_{r+2} (\mathbf{e}_{c}^{n+m}-\mathbf{e}_{d}^{n+m})=\mathbf{0},
  $$
which implies that $c,d\in V(G)$ or $c,d\in I(G)$.

Suppose that $c,d\in I(G)$. By (\ref{F42}) and (\ref{F44}), we have
\begin{align}\label{F528}
     F_{\theta^{\pm}}\left ( \mathbf{e}_{c}^{m+n}-\mathbf{e}_{d}^{m+n}\right ) =&\frac{\left (\theta+2-\theta^{\pm}  \right )^2 }{\left (\theta+2-\theta^{\pm}  \right
     )^2+\left (2r-\theta \right )}
 \left(
	\begin{array}{c}
 \frac{1}{\theta+2-\theta^{\pm}} F_{\theta}(G)R_{G}\left ( \mathbf{e}_{c}^{m}-\mathbf{e}_{d}^{m}\right )\\ [0.3cm]
  \frac{1}{\left (\theta+2-\theta^{\pm}\right )^{2}}R_{G}^{\top }F_{\theta}(G)R_{G}\left ( \mathbf{e}_{c}^{m}-\mathbf{e}_{d}^{m}\right )
	\end{array}
	\right).
	\end{align}
By (\ref{Q1}), (\ref{Eigen73}) and (\ref{F528}), we have
\begin{align*}
&(\theta+2-\theta^{+})F_{\theta}(G)(\mathbf{e}_{a}^{n}-\mathbf{e}_{b}^{n})=\pm F_{\theta}(G)R_{G}(\mathbf{e}_{c}^{m}-\mathbf{e}_{d}^{m}),\\[0.2cm]
&(\theta+2-\theta^{-})F_{\theta}(G)(\mathbf{e}_{a}^{n}-\mathbf{e}_{b}^{n})=\pm F_{\theta}(G)R_{G}(\mathbf{e}_{c}^{m}-\mathbf{e}_{d}^{m}).
\end{align*}

If both $\theta^{\pm}\in \Lambda_{ab,cd}^{+}$ or $\theta^{\pm}\in \Lambda_{ab,cd}^{-}$, then $\theta^{+}=\theta^{-}$,
a contradiction to (\ref{F441}) and (\ref{F441--1}).

If $\theta^{+}\in \Lambda_{ab,cd}^{+}$ and $\theta^{-}\in \Lambda_{ab,cd}^{-}$, or  $\theta^{+}\in \Lambda_{ab,cd}^{-}$ and $\theta^{-}\in \Lambda_{ab,cd}^{+}$, then
$$
\theta^{+}+\theta^{-}=2(\theta+2).
$$
By (\ref{F441}) and (\ref{F441--1}),  we have $\theta=r-2$. So, $$\mathrm{{supp}}_{L_{\mathcal{Q} \left ( G \right )}}(\mathbf{e}_{a}^{m+n}-\mathbf{e}_{b}^{m+n})=\left
\{(r-2)^{\pm}  \right \},$$
and by (\ref{PMandG--1}),
\begin{equation}\label{GSupp--21}
\mathrm{{supp}}_{L\left ( G \right )}(\mathbf{e}_{a}^{n}-\mathbf{e}_{b}^{n})=\left \{ r-2  \right \}.
\end{equation}
Recall that the sum of all eigenprojectors of $G$ is equal to the identity matrix. By the definition of Laplacian eigenvalue support and (\ref{GSupp--21}), we have
\begin{align}\label{Q4}
F_{r-2}(G)(\mathbf{e}_{a}^{n}-\mathbf{e}_{b}^{n})=\mathbf{e}_{a}^{n}-\mathbf{e}_{b}^{n}.
\end{align}
By (\ref{spect1}), (\ref{GSupp--21}) and (\ref{Q4}), we have
$$
L_G(\mathbf{e}_{a}^{n}-\mathbf{e}_{b}^{n})=\sum_{r=0}^{p}\theta_r
F_{\theta_r}(G)(\mathbf{e}_{a}^{n}-\mathbf{e}_{b}^{n})=(r-2)F_{r-2}(G)(\mathbf{e}_{a}^{n}-\mathbf{e}_{b}^{n})=(r-2)(\mathbf{e}_{a}^{n}-\mathbf{e}_{b}^{n}),
$$
which can not occur. Thus, we must have $c,d\in V(G)$.

For $c,d\in V(G)$, similar to (\ref{Q1}), we have
\begin{align}\label{Qcd--2121}
 	F_{\theta^{\pm}}\left ( \mathbf{e}_{c}^{m+n}-\mathbf{e}_{d}^{m+n}\right )=\frac{\left (\theta+2-\theta^{\pm}  \right )^2 }{\left (\theta+2-\theta^{\pm}  \right )^2+\left
(2r-\theta \right )}\begin{pmatrix}
F_{\theta}(G)(\mathbf{e}_{c}^{n}-\mathbf{e}_{d}^{n}) \\[0.2cm] \frac{1}{\theta+2-\theta^{\pm}} \left (  F_{\theta}(G)R_{G} \right )^{\top
}(\mathbf{e}_{c}^{n}-\mathbf{e}_{d}^{n})
\end{pmatrix}.
	\end{align}
If $\theta^{+} \in \Lambda^{+}_{ab,cd}$, that is,
$$
F_{\theta^{+}}\left ( \mathbf{e}_{a}^{m+n}-\mathbf{e}_{b}^{m+n}\right )=F_{\theta^{+}}\left ( \mathbf{e}_{c}^{m+n}-\mathbf{e}_{d}^{m+n}\right ),
$$
by (\ref{Q1}) and (\ref{Qcd--2121}), we then have
\begin{equation}\label{FinalEqu12121}
F_{\theta}\left ( \mathbf{e}_{a}^{n}-\mathbf{e}_{b}^{n}\right )=F_{\theta}\left ( \mathbf{e}_{c}^{n}-\mathbf{e}_{d}^{n}\right ).
\end{equation}
Together with (\ref{Q1}) and (\ref{Qcd--2121}), (\ref{FinalEqu12121}) leads to
$$
F_{\theta^{-}}\left ( \mathbf{e}_{a}^{m+n}-\mathbf{e}_{b}^{m+n}\right )=F_{\theta^{-}}\left ( \mathbf{e}_{a}^{m+n}-\mathbf{e}_{b}^{m+n}\right ).
$$
Thus, $\theta^{+} \in \Lambda^{+}_{ab,cd}$ if and only if $\theta^{-} \in \Lambda^{+}_{ab,cd}$. Similarly, we can get that $\theta^{+} \in \Lambda^{-}_{ab,cd}$ if and only if
$\theta^{-} \in \Lambda^{-}_{ab,cd}$.

\noindent\emph{Case 2.} $a, b \in I(G)$. The proof is similar to that of Case 1, and hence we omit the details here.
\qed
\end{proof}

\begin{theorem}\label{hpest}
Let $G$ be an $r$-regular connected graph with $n$ vertices and $m$ edges. If $r+1$ is prime, then $\mathcal{Q} \left ( G \right )$ does not have Pair-LPST.
\end{theorem}
\begin{proof}
We prove the result by contradiction. Let $\mathbf e_{a}^{m+n}-\mathbf e_{b}^{m+n} $ and $\mathbf e_{c}^{m+n}-\mathbf e_{d}^{m+n} $ be two distinct pair states of $\mathcal{Q}
\left ( G \right )$, and suppose that $\mathcal{Q} \left ( G \right )$ has Pair-LPST between $\mathbf e_{a}^{m+n}-\mathbf e_{b}^{m+n} $ and $\mathbf e_{c}^{m+n}-\mathbf
e_{d}^{m+n}$. Let
$$
S=\mathrm{{supp}}_{L_{\mathcal{Q}(G)}}\left( \mathbf e_{a}^{m+n}-\mathbf e_{b}^{m+n} \right ).
$$
Consider the following cases.

\noindent\emph{Case 1.}
$G$ is a non-bipartite graph.

\noindent\emph{Case 1.1.} $a, b \in I(G)$. By Theorem \ref{C_n}, we have
 $$
 S\subseteq \left \{ \theta_{i}^{\pm } \right \} \cup \left \{ 2r+2 \right \},~0\le i\le p,
 $$
and $\theta_{i}^{+}\in S$ if and only if $\theta_{i}^{-}\in S$, where $\theta_0> \theta_1>\cdots>\theta_p=0 $ are all distinct Laplacian eigenvalues of $G$. Recall Lemma \ref{Coutinho} (b) that the eigenvalues in $S$ are either all integers
or all quadratic integers. Consider the following cases.

\noindent\emph{Case 1.1.1.}
All the eigenvalues in $S$ are integers. In this case, if there exist $\theta_{i}^{\pm}\in S$, by (\ref{F441}), both
$$
\theta_{i}^{+ }+\theta_{i}^{- }=2+r+\theta_{i},
$$
and
$$
\theta_{i}^{+ }-\theta_{i}^{- }=\sqrt{\left ( r+2-\theta_{i}  \right )^2+4\theta_{i}  },
$$
are integers, which implies that $\theta_{i}$ is an integer and $\left ( r+2-\theta_{i}  \right )^2+4\theta_{i}$ is a perfect square. Since $4\theta_{i}$ is even, the parity
of $\left (r+2-\theta_{i}  \right )^2+4\theta_{i}$ and $\left ( r+2-\theta _{i} \right )^2$ must be same. Thus, there exists an integer $\delta\left ( \ge 1 \right )$ satisfying
$$
\left ( r+2-\theta_{i}  \right )^2+4\theta_{i}=\left ( r-\theta_{i}+2\delta  \right )^2
$$
which implies that
$$
\theta_{i} =\frac{\left ( \delta-1 \right ) \left ( \delta+1+r \right ) }{\delta} =\delta+r-\frac{1+r}{\delta}.
$$
Since $r+1$ is prime, we have that $\delta=r+1$ or $\delta=1$.

If $\delta=r+1$, then $\theta_{i}=2r$, which contradicts the fact that $G$ is an $r$-regular non-bipartite graph.

If $\delta=1$, then $\theta_{i}=0$. By (\ref{F441}) and (\ref{F42}), we have
\begin{align}\label{537}
F_{0^-}=F_{0}=\frac{1}{n+m}J_{n+ m}.
\end{align}
Together with (\ref{Eigen72}) and (\ref{537}), we have
$$
F_{0}(\mathbf e_{a}^{m+n}-\mathbf e_{b}^{m+n})=\mathbf{0},
$$
and
$$
F_{r+2}(\mathbf e_{a}^{m+n}-\mathbf e_{b}^{m+n})=\mathbf{0},
$$
a contradiction to $\theta_{i}^{\pm}\in S$. Therefore, $S=\left \{ 2r+2 \right \}$. By Lemma
\ref{support}, $\mathbf{e}_{a}^{m+n}-\mathbf{e}_{b}^{m+n}$ does not have Pair-LPST in this case.

\noindent\emph{Case 1.1.2.}
All the eigenvalues in $S$ are quadratic integers. For $\theta_{i}^{\pm} \in S$, assume that
\begin{align}\label{Eigen5}
\theta_{i}^{\pm}&=\frac{1}{2}(\rho+\eta _{\pm}\sqrt{\Delta}).
\end{align}
Then
\begin{align}\label{Eigen11-11}
 4\theta_{i}^{+} \theta_{i}^{-} =\rho^{2}+\eta_{+}\eta_{-}\Delta+\rho\left (\eta_++\eta_-  \right )\sqrt{\Delta}
 \end{align}
By $\left (\ref{F441} \right )$ and $\left (\ref{Eigen5} \right )$, we have
\begin{align}\label{Eigen8}
 \theta_{i}=\rho+\frac{\eta_++\eta_-}{2}\sqrt{\Delta}-2-r.
\end{align}
Then $\left (\ref{F441} \right )$ and (\ref{Eigen8}) imply that
\begin{align}\label{Eigen11}
 4\theta_{i}^{+} \theta_{i}^{-}
 &=4\left (\rho-r-2  \right )\left (r+1  \right )+2\left (\eta_++\eta_-  \right )\left (r+1  \right )\sqrt{\Delta}.
\end{align}
Since $\sqrt{\Delta}$ is irrational, by (\ref{Eigen11-11}) and (\ref{Eigen11}), we have
$$
\rho^{2}+\eta_{+}\eta_{-}\Delta=4\left (\rho-r-2  \right )\left (r+1  \right ),
$$
and
$$
\rho\left (\eta_++\eta_-  \right )\sqrt{\Delta}=2\left (\eta_++\eta_-  \right )\left (r+1  \right )\sqrt{\Delta}.
$$
Thus
\begin{align}\label{522}
\eta_{+}\eta_{-}\Delta&=-4\left (r+1\right ), \text{~and~}\rho=2r+2.
\end{align}

Notice Lemma \ref{Coutinho} (b) and $\left (\ref{Eigen5} \right )$ that
$$
\theta_{i}^{+}-\theta_{i}^{-}=\frac{1}{2}\left ( \eta_{+}-\eta_{-} \right ) \sqrt{\Delta}
$$
is an integer multiple of $\sqrt{\Delta}$. Thus, $\eta_{+}$ and $\eta_{-}$ have the same parity.

If both $\eta_+$ and $\eta_-$ are even, then (\ref{522}) leads to $\Delta \mid r+1$. Recall that $r+1$ is prime. Then we have
$$
\Delta=r+1,\ \eta_{+}=-\eta_{-}=2.
$$
By $\left (\ref{Eigen8}\right )$ and (\ref{522}), we have $\theta_{i}=r$, and then $S=\left \{r^{\pm}  \right \} $. By Theorem \ref{C_n}, we have
 $$
 r^{\pm}\in \Lambda^{+}_{ab,cd}\text{~or~} r^{\pm}\in \Lambda^{-}_{ab,cd},
 $$
 which means that one of $\Lambda^{\pm}_{ab,cd}$ is empty. This can not occur. Thus, $\mathcal{Q} \left ( G \right )$ does not have Pair-LPST between $\mathbf{e}_{a}^{m+n}-\mathbf{e}_{b}^{m+n}$
 and $\mathbf{e}_{c}^{m+n}-\mathbf{e}_{d}^{m+n}$ in this case.

If both $\eta_+$ and $\eta_-$ are odd, by (\ref{522}), we have
$$
\Delta= -\frac{4(r+1)}{\eta_+\eta_-},
$$
which implies that $ 4\mid \Delta$, a contradiction to the fact that $\Delta$ is a square-free number. Thus, $\mathcal{Q} \left ( G \right )$ does not have Pair-LPST between
$\mathbf{e}_{a}^{m+n}-\mathbf{e}_{b}^{m+n}$ and $\mathbf{e}_{c}^{m+n}-\mathbf{e}_{d}^{m+n}$ in this case.

\noindent\emph{Case 1.2.} $a \in V(G), b \in I(G)$.
By (\ref{Eigen72}), we have
\begin{align}\label{538}
F_{0^{+}}\left( \mathbf e_{a}^{m+n}-\mathbf e_{b}^{m+n} \right )=F_{r+2}\left( \mathbf e_{a}^{m+n}-\mathbf e_{b}^{m+n} \right )\neq\mathbf{0},
\end{align}
which means $r+2\in S$, and then all the eigenvalues in $S$ should be integers. By Lemma \ref{support}, if $\mathbf e_{a}^{m+n}-\mathbf e_{b}^{m+n}$ has Pair-LPST, then there exist at least two eigenvalues in $S$.
Assume that $\kappa \in S$ is an integer such that $\kappa \not= r+2$.
Without loss of generality, say $\theta_{i}^{+}=\kappa $. By (\ref{F441}), we have
$$
\theta_i =\frac{\kappa }{1+r-\kappa }+\kappa ,
$$
and then
\begin{align}\label{532}
\theta_{i}^{-}&=\frac{\kappa }{1+r-\kappa }+r+2.
\end{align}
Note (\ref{532}) that $\theta_{i}^{-}$ is a rational Laplacian eigenvalue of $\mathcal{Q} \left ( G \right )$. Thus,  $\theta_{i}^{-}$ is an integer. Similarly, if $\theta_{i}^{-}\in S$, we can also get that $\theta_{i}^{+}$ is an integer.  Thus, if $\theta_{i}^{+}\in S$ or $\theta_{i}^{-}\in S$ is an integer, then $\theta_{i}^{\pm}$ are both integers. Similar to the discussion of Case 1.1.1, we have $\theta_{i}=0$. Thus,
$$
S\subseteq \left \{ 0^\pm,2r+2 \right \}.
$$
By (\ref{537}), we have
$$
F_{0^-}(\mathbf e_{a}^{m+n}-\mathbf e_{b}^{m+n})=F_{0}(\mathbf e_{a}^{m+n}-\mathbf e_{b}^{m+n})=\mathbf{0}.
$$
Thus, $0^-\notin S$. By (\ref{F41}) and (\ref{538}), we have
$$
S\subseteq \left \{ r+2,2r+2 \right \}.
$$
Recall that the sum of all eigenprojectors of $\mathcal{Q}(G)$ is equal to the identity matrix. By the definition of Laplacian eigenvalue support, we have
\begin{align}\label{604}
(F_{r+2}+F_{2r+2})(\mathbf e_{a}^{m+n}-\mathbf e_{b}^{m+n})=\mathbf e_{a}^{m+n}-\mathbf e_{b}^{m+n}.
\end{align}
Left multiplying $(\mathbf e_{a}^{n+m })^{\top}$ on both sides of (\ref{604}), by (\ref{F41}) and (\ref{Eigen72}), we obtain that the equality of (\ref{604}) holds if and only if $n=1$, a contradiction to $r\ge2$. Thus, $\mathcal{Q} \left ( G \right )$ does not have Pair-LPST between $\mathbf{e}_{a}^{m+n}-\mathbf{e}_{b}^{m+n}$ and
  $\mathbf{e}_{c}^{m+n}-\mathbf{e}_{d}^{m+n}$ in this case.

\noindent\emph{Case 1.3.} $a,b \in V(G)$.
The proof is similar to that of Case 1.1,  and hence we omit the details here.

\noindent\emph{Case 2.} $G$ is a bipartite graph.

\noindent\emph{Case 2.1.} $a , b \in I(G)$.
Since $G$ is an $r$-regular bipartite graph, we have $\theta_{0}=2r \in \mathrm{{Spec}}_{L}(G)$.
By Lemma \ref{eigenvalues} and Theorem \ref{C_n}, we have
\begin{align*}
S\subseteq \left \{ \theta_{i}^{\pm } \right \}\cup \left \{ r,2r+2 \right \},~ 1\le i\le p,
\end{align*}
and for each $\theta_{i}\neq2r$, $\theta_{i}^{+}\in S$ if and only if $\theta_{i}^{-}\in S$, where $ \theta_1>\cdots>\theta_p=0 $ are all distinct Laplacian eigenvalues of $G$ except for $2r$. Similar to Case 1.1, we consider the following cases.

\noindent\emph{Case 2.1.1.}
All the eigenvalues in $S$ are integers. Recall the discussion of Case 1.1.1 that if $\theta_{i}^{\pm}$ are both integers, then $\theta_{i}=0$ or $\theta_{i}=2r$. By (\ref{Eigen72}) and (\ref{537}), we have
$$
F_{0^-}(\mathbf e_{a}^{m+n}-\mathbf e_{b}^{m+n})=F_{0^+}(\mathbf e_{a}^{m+n}-\mathbf e_{b}^{m+n})=\mathbf{0},
$$
which implies that $0^\pm\notin S$. Note that $(2r)^+=2r+2$ and $(2r)^-=r$. By (\ref{F45}), we have
$$
F_{r}(\mathbf e_{a}^{m+n}-\mathbf e_{b}^{m+n})=\mathbf{0}.
$$
Thus
$$
S\subseteq\left \{ 2r+2 \right \}.
$$
By Lemma \ref{support}, we can conclude that $\mathbf e_{a}^{m+n}-\mathbf e_{b}^{m+n}$ does not have Pair-LPST in this case.

\noindent\emph{Case 2.1.2.}
All the eigenvalues in $S$ are quadratic integers. The proof is similar to Case 1.1.2,  and hence we omit the details here.

\noindent\emph{Case 2.2.} $a \in V(G), b \in I(G)$.
By (\ref{Eigen72}), we have
$$
F_{0^{+}}\left( \mathbf e_{a}^{m+n}-\mathbf e_{b}^{m+n} \right )=F_{r+2}\left( \mathbf e_{a}^{m+n}-\mathbf e_{b}^{m+n} \right )\neq\mathbf{0},
$$
which means $r+2\in S$, and then all the eigenvalues in $S$ should be integers.  According to the discussion of Case 1.2, if $\theta_{i}^{+}\in S$ or $\theta_{i}^{-}\in S$ is an integer, then $\theta_{i}^{\pm}$ are both integers. Recall that if $\theta_{i}^{\pm}$ are both integers, then $\theta_{i}=0$ or $\theta_{i}=2r$. By Lemma \ref{eigenvalues}, we have
 $$
S\subseteq \left \{0^{\pm},r,2r+2 \right \} .
$$
Note (\ref{537}) that
$$
F_{0^-}(\mathbf e_{a}^{m+n}-\mathbf e_{b}^{m+n})=\mathbf{0},
$$
which implies
 $$
S\subseteq \left \{0^{+}, r,2r+2\right \}.
$$
Since the sum of all eigenprojectors of $\mathcal{Q}(G)$ is equal to the identity matrix, according to the definition of eigenvalues support, we have
\begin{align}\label{531}
(F_{0^{+}}+F_{r}+F_{2r+2})(\mathbf e_{a}^{m+n}-\mathbf e_{b}^{m+n})=\mathbf e_{a}^{m+n}-\mathbf e_{b}^{m+n}.
\end{align}
Left multiplying $\mathbf (\mathbf e_{a}^{n+m})^ {\top }$ on both sides of (\ref{531}), by (\ref{F43}), (\ref{F45}) and (\ref{Eigen72}), we get that the equality of (\ref{531}) holds if and only if $n=2$, a contradiction to $r\ge2$.
Thus, $\mathcal{Q} \left ( G \right )$ does not have Pair-LPST between $\mathbf{e}_{a}^{m+n}-\mathbf{e}_{b}^{m+n}$ and $\mathbf{e}_{c}^{m+n}-\mathbf{e}_{d}^{m+n}$.

\noindent\emph{Case 2.3. } $a , b \in V(G)$.

\noindent\emph{Case 2.3.1.} $a \in V_{1}$, $b \in V_{2}$.
By (\ref{F45}),
 $$
F_{(2r)^-}\left( \mathbf e_{a}^{m+n}-\mathbf e_{b}^{m+n} \right )=F_{r}\left( \mathbf e_{a}^{m+n}-\mathbf e_{b}^{m+n} \right )\ne\mathbf{0},
 $$
which means $r\in S$, and then all the eigenvalues in $S$ should be integers. According to Theorem \ref{C_n}, we have that $\theta_{i}^{+}\in S$ if and only if $\theta_{i}^{-}\in S$. Recall that if $\theta_{i}^{\pm}$ both are integers, then $\theta_{i}=0$ or $\theta_{i}=2r$.
Then
 $$
S\subseteq \left \{0^{\pm},(2r)^{\pm} \right \}= \left \{ 0,r+2, r,2r+2 \right \}.
$$
By (\ref{Eigen72}) and (\ref{537}), it is easy to verify that $0^{\pm}\notin S$. By (\ref{F43}) and (\ref{F45}), we have
$$
F_{2r+2}( \mathbf e_{a}^{m+n}-\mathbf e_{b}^{m+n})=\mathbf{0},
$$
and
 $$
 F_{r}( \mathbf e_{a}^{m+n}-\mathbf e_{b}^{m+n})\neq\mathbf{0}.
 $$
 Thus, $S=\left \{r\right\}$. By Lemma \ref{support}, $\mathcal{Q} \left ( G \right )$ does not have Pair-LPST between $\mathbf{e}_{a}^{m+n}-\mathbf{e}_{b}^{m+n}$ and
 $\mathbf{e}_{c}^{m+n}-\mathbf{e}_{d}^{m+n}$ in this case.

\noindent\emph{Case 2.3.2.} $a,b \in V_{i}, i=1,2$. By Lemma \ref{eigenvalues} and Theorem \ref{C_n}, we know that
 $$
 S\subseteq \left \{ \theta_{i}^{\pm } \right \} ,~ 1\le i\le p,
 $$
 where $ \theta_1>\cdots>\theta_p=0 $ are all distinct Laplacian eigenvalues of $G$ except for $2r$. 
If all the eigenvalues in $S$ are integers, by refining the proof of Case 2.3.1,  one can easily verify that $S=\emptyset$, which can not occur.
If all the eigenvalues in $S$ are quadratic integers, similar to Case 1.1.2, we get the contradiction.
 Thus, $\mathcal{Q} \left ( G \right )$ does not have Pair-LPST between $\mathbf{e}_{a}^{m+n}-\mathbf{e}_{b}^{m+n}$ and
 $\mathbf{e}_{c}^{m+n}-\mathbf{e}_{d}^{m+n}$ in this case. 
 
This completes the proof. \qed
\end{proof}

Let $C_{n}$ denote a cycle with $n$ vertices. Since $C_{n}$ is a $2$-regular graph and $2+1=3$ is prime, Theorem \ref{hpest} implies the following result immediately.
\begin{cor} \label{noperfect}
$\mathcal{Q} \left ( C_{n} \right )$ does not  have Pair-LPST.
\end{cor}

\begin{theorem}\label{2p}
 Let $G$ be an $r$-regular connected graph with $n$ vertices and $m$ edges. If $r+1=2^{t}$ is a power of 2 with $t\ge2$, then $\mathcal{Q} \left ( G \right )$ does not have
 Pair-LPST.
\end{theorem}
\begin{proof}
We prove the result by contradiction. Let $\mathbf e_{a}^{m+n}-\mathbf e_{b}^{m+n} $ and $\mathbf e_{c}^{m+n}-\mathbf e_{d}^{m+n} $ be two distinct pair states of $\mathcal{Q}
\left ( G \right )$, and suppose that $\mathcal{Q} \left ( G \right )$ has Pair-LPST between $\mathbf e_{a}^{m+n}-\mathbf e_{b}^{m+n} $ and $\mathbf e_{c}^{m+n}-\mathbf
e_{d}^{m+n}$. Let
$$
S=\mathrm{{supp}}_{L_{\mathcal{Q}(G)}}\left( \mathbf e_{a}^{m+n}-\mathbf e_{b}^{m+n} \right ).
$$
Consider the following cases.

\noindent\emph{Case 1.} $G$ is a non-bipartite graph.

\noindent\emph{Case 1.1.} $a,b\in V(G)$.

\noindent\emph{Case 1.1.1.} All the eigenvalues in $S$ are integers. According to the discussion of Case 1.1 in Theorem \ref{hpest}, for each $\theta \in
\mathrm{{Spec}}_{L}(G)$, if $\theta^{\pm}$ are both integers, then
\begin{align}\label{2p1}
\theta =\frac{\left ( \delta-1 \right ) \left (\delta+1+r \right ) }{\delta} =\delta+2^{t}-1-\frac{2^{t}}{\delta}, ~(\delta \ge1).
 \end{align}
Note (\ref{F41}) that $2r+2\notin S$, and recall that neither of $\Lambda^{\pm}_{ab,cd}$ can be empty. By Theorem \ref{C_n}, we assume $\theta_{q}^{\pm}\in \Lambda^{+}_{ab,cd}$ and
 $\theta_{s}^{\pm}\in \Lambda^{-}_{ab,cd}$. Notice that both $\theta_{q}^{\pm}$ and $\theta_{s}^{\pm}$ are integers. By (\ref{F441}), we have both $ \theta_{q}$ and  $\theta_{s}$ are integers. By (\ref{2p1}), assume that
 $$
 \theta_{q}=2^{\alpha }+2^{t}-1-2^{t-\alpha },~
 \theta_{s}=2^{\beta }+2^{t}-1-2^{t-\beta },~(\alpha\neq \beta ).
 $$
By (\ref{F441}), we have
\begin{align*} 
 \theta_{q}^{+}&=2^{t}+2^{\alpha},~ \theta_{q}^{-}=2^{t}-2^{t-\alpha},\\[0.1cm]
 \theta_{s}^{+}&=2^{t}+2^{\beta},~ \theta_{s}^{-}=2^{t}-2^{t-\beta}.
 \end{align*}
According to Lemma \ref{Coutinho} (c), we know that $g\mid \left(\theta_{q}^{\pm}-\theta_{s}^{\pm}\right)$ and $\frac{\theta_{q}^{\pm}-\theta_{s}^{\pm}}{g}$ are odd. Without loss of generality, we assume that $\alpha>\beta$. Suppose that $g=2^xy$ with $y$ is an odd integer. Since
\begin{align*}
\frac{\theta_{q}^{+}-\theta_{s}^{+}}{g}=\frac{2^\alpha-2^\beta}{2^xy} \text{~and~}\frac{\theta_{q}^{-}-\theta_{s}^{-}}{g}=\frac{2^{t-\beta}-2^{t-\alpha}}{2^xy},
 \end{align*}
are odd, we have
\begin{align}\label{606}
x=\min\left\{\alpha,\beta\right\}=\beta=\min\left\{t-\beta,t-\alpha\right\}=t-\alpha.
 \end{align}
According to Lemma \ref{Coutinho} (c), by (\ref{606}), we have
$$
\frac{\theta_{q}^{+}-\theta_{s}^{-}}{g}=\frac{2^\alpha+2^{t-\beta}}{2^xy}=\frac{2^{\alpha+1}}{2^xy},
$$
is odd, which implies that $x=\alpha+1$. By (\ref{606}), we have $\alpha+1=\beta$, a contradiction to $\alpha>\beta$.

 \noindent\emph{Case 1.1.2.} All the eigenvalues in $S$ are quadratic integers. For $\theta_i^{\pm} \in S$, let
\begin{align*}
\theta_i^{\pm}=\frac{1}{2}(\rho +\eta _{\pm}\sqrt{\Delta}).
\end{align*}
By (\ref{522}), we have
$$
\eta _{+}\eta _{-}\Delta=-4\left (r+1\right )=-2^{t+2}, \text{~and~}\rho=2r+2=2^{t+1}.
$$
Since $\Delta$ is a square-free integer, we have $\Delta=2$. Thus
$$
\eta _{+}\eta _{-}=-2^{t+1}.
$$
Then we can assume that $\eta _{+}=2^{i_{+}}$,  $\eta _{-}=-2^{i_{-}}$ with $i_{+}+i_{-}=t+1$. Recall that neither of $\Lambda^{\pm}_{ab,cd}$ can be empty. By Theorem \ref{C_n}, assume that $\theta_{q}^{\pm}\in \Lambda^{+}_{ab,cd}$ and $\theta_{s}^{\pm}\in \Lambda^{-}_{ab,cd}$ with
 \begin{align*}
 \theta_{q}^{+}=2^{t}+2^{q_+}\sqrt{2},~ \theta_{q}^{-}=2^{t}-2^{q_-}\sqrt{2},~\text{with $q_++q_-=t+1$;}\\[0.1cm]
 \theta_{s}^{+}=2^{t}+2^{s_+}\sqrt{2},~ \theta_{s}^{-}=2^{t}-2^{s_-}\sqrt{2},~\text{with $s_++s_-=t+1$.}
 \end{align*}
By Lemma \ref{Coutinho} (c), we know that $g\mid \frac{\theta_{q}^{\pm}-\theta_{s}^{\pm}}{\sqrt{2}}$ and $\frac{\theta_{q}^{\pm}-\theta_{s}^{\pm}}{\sqrt{2}g}$
are odd. Without loss of generality, we assume that $q_+>s_+$. Suppose that $g=2^{x'}y'$ with $y'$ is an odd integer. Since
\begin{align*}
\frac{\theta_{q}^{+}-\theta_{s}^{+}}{g\sqrt{2}}=\frac{2^{q_+}-2^{s_+}}{2^{x'}y'} \text{~and~}\frac{\theta_{q}^{-}-\theta_{s}^{-}}{g\sqrt{2}} =\frac{2^{s_-}-2^{q_-}}{2^{x'}y'},
 \end{align*}
are odd, we have
\begin{align*}
x'=\min\left\{q_+,s_+\right\}=s_+=\min\left\{q_-,s_-\right\}=t+1-q_+.
 \end{align*}
Thus
$$
\frac{\theta_{q}^{+}-\theta_{s}^{-}}{g\sqrt{2}}=\frac{2^{q_+}+2^{s_-}}{2^{x'}y'} =\frac{2^{q_+}+2^{q_+}}{2^{x'}y'}=\frac{2^{q_++1}}{2^{x'}y'},
$$
which can not be odd. This contradicts Lemma \ref{Coutinho} (c).

\noindent\emph{Case 1.2.} $a,b\in I(G)$.

\noindent\emph{Case 1.2.1.} All the eigenvalues in $S$ are integers.  By (\ref{F41}) and Theorem \ref{C_n}, we know that
 $$
 S\subseteq \left \{ \theta_{i}^{\pm } \right \} \cup \left \{ 2r+2\right \}, ~0\le i\le p,
 $$
and for each $\theta_{i}$, $\theta_{i}^{+}\in S$ if and only if $\theta_{i}^{-}\in S$, where $\theta_0> \theta_1>\cdots>\theta_p=0 $ are all distinct Laplacian eigenvalues of $G$.

If $2r+2\notin S$, by Theorem \ref{C_n}, we assume $\theta_{q}^{\pm}\in \Lambda^{+}_{ab,cd}$ and
 $\theta_{s}^{\pm}\in \Lambda^{-}_{ab,cd}$, since neither of $\Lambda^{\pm}_{ab,cd}$ can be empty. Similar to Case 1.1.1, we can get the contradiction.

If $2r+2\in S$, without loss of generality, assume that $2r+2\in \Lambda_{ab,cd}^{+}$. Note that $\Lambda_{ab,cd}^{-}$ can not be empty.
Then, we assume $\theta^{\pm}\in \Lambda_{ab,cd}^{-}$.
Since $\theta^{\pm}$ are both integers, by (\ref{F441}), $ \theta$ is also an integer. Similar to the discussion of Case 1.1.1, by (\ref{2p1}), assume that
$$
 \theta=2^{k}+2^{t}-1-2^{t-k}, \text{~where~}0\le k < t.
$$
By (\ref{F441}), we have
$$
 \theta^{+}=2^{t}+2^{k},~ \theta^{-}=2^{t}-2^{t-k}.
$$
According to Lemma \ref{Coutinho} (c), we have that
$$
\frac{2r+2-\theta^{\pm}}{g}=\frac{2^{t+1}-\theta^{\pm}}{2^{x''}y''}
$$
are odd. Then
$$
x''=\min\left\{k,t\right\}=\min\left\{t-k,t\right\},
$$
which implies that
$$
 k=\frac{t}{2}~ \text{and} ~\theta=2^{t}-1=r.
$$
Thus
\begin{align} \label{lambda-}
\Lambda_{ab,cd}^{-}= \left \{ r^{\pm } \right \}.
\end{align}
Recall that $2r+2\in \Lambda_{ab,cd}^{+}$.  Assume that, besides $2r+2$, $\theta_{q}^{\pm}\in \Lambda^{+}_{ab,cd}$. Similar to the discussion of Case 1.1.1, we can get the contradiction.
Thus
$$
\Lambda_{ab,cd}^{+}= \left \{2r+2 \right \},
$$
and
\begin{align} \label{nonSupport}
S= \left \{r^{\pm}, 2r+2  \right \}.
\end{align}
Recall that the sum of all eigenprojectors of $\mathcal{Q}(G)$ is equal to the identity matrix. By the definition of Laplacian eigenvalue support, we have
\begin{align}\label{eigen support}
\left\{\begin{matrix}
\left ( \sum_{\gamma\in\Lambda^{+}_{ab,cd}}F_{\gamma}+ \sum_{\gamma\in\Lambda^{-}_{ab,cd}}F_{\gamma}
\right)(\mathbf{e}_{a}^{n+m}-\mathbf{e}_{b}^{n+m})=\mathbf{e}_{a}^{n+m}-\mathbf{e}_{b}^{n+m}, \\[0.4cm]
\left ( \sum_{\gamma\in\Lambda^{+}_{ab,cd}}F_{\gamma}+ \sum_{\gamma\in\Lambda^{-}_{ab,cd}}F_{\gamma}
\right)(\mathbf{e}_{c}^{n+m}-\mathbf{e}_{d}^{n+m})=\mathbf{e}_{c}^{n+m}-\mathbf{e}_{d}^{n+m}.
\end{matrix}\right.
\end{align}
Then (\ref{eigen support}) leads to
\begin{align}
\label{GP33} \sum_{\gamma\in\Lambda^{-}_{ab,cd}}F_{\gamma}(\mathbf{e}_{a}^{n+m}-\mathbf{e}_{b}^{n+m})
=&\frac{\mathbf{e}_{a}^{n+m}-\mathbf{e}_{b}^{n+m}-\mathbf{e}_{c}^{n+m}+\mathbf{e}_{d}^{n+m}}{2}, \\
\label{gam+} \sum_{\gamma\in\Lambda^{+}_{ab,cd}}F_{\gamma}(\mathbf{e}_{a}^{n+m}-\mathbf{e}_{b}^{n+m})
=&\frac{\mathbf{e}_{a}^{n+m}-\mathbf{e}_{b}^{n+m}+\mathbf{e}_{c}^{n+m}-\mathbf{e}_{d}^{n+m}}{2}.
\end{align}
By (\ref{lambda-}) and (\ref{GP33}), we have
  \begin{align} \label{GP1}
(F_{r^{+}}+F_{r^{-}}) (\mathbf{e}_{a}^{n+m}-\mathbf{e}_{b}^{n+m})=\frac{\mathbf{e}_{a}^{n+m}-\mathbf{e}_{b}^{n+m}-\mathbf{e}_{c}^{n+m}+\mathbf{e}_{d}^{n+m}}{2}.
	\end{align}

On the other hand, according to (\ref{F441}), we have
\begin{align}\label{Fsum}
\frac{1 }{ (\theta+2-\theta^{+})^2+
(2r-\theta)}
+\frac{1 }{ (\theta+2-\theta^{-})^2+
(2r-\theta)}=\frac{1}{2r-\theta},
\end{align}
and
\begin{align}\label{Fp=0}
\frac{\theta+2-\theta^{+}    }{\left (\theta+2-\theta^{+}  \right )^2+\left
(2r-\theta \right )}
+\frac{\theta+2-\theta^{-}   }{\left (\theta+2-\theta^{-}  \right )^2+\left
(2r-\theta \right )}=0.
\end{align}
Then, by (\ref{F42}), (\ref{Fsum}) and (\ref{Fp=0}), we get
\begin{align}\label{GP77}
    (F_{r^{+}}+F_{r^{-}}) (\mathbf{e}_{a}^{n+m}-\mathbf{e}_{b}^{n+m}) =&
 \left(
	\begin{array}{cc}
\mathbf{0}\\ [0.3cm]
	   \frac{1}{r }R_{G}^{\top }F_{r}(G)R_{G}(\mathbf{e}_{a}^{m}-\mathbf{e}_{b}^{m})
	\end{array}
	\right).
\end{align}
Combining (\ref{GP1}) and (\ref{GP77}), we know that $c,d\in I (G)$ and
\begin{align}\label{Fr}
 R_{G}^{\top }F_{r}(G)R_{G}(\mathbf{e}_{a}^{m}-\mathbf{e}_{b}^{m})
 =\frac{r(\mathbf{e}_{a}^{m}-\mathbf{e}_{b}^{m}-\mathbf{e}_{c}^{m}+\mathbf{e}_{d}^{m})}{2}.
\end{align}
 Notice (\ref{nonSupport}) that $S= \left \{r^{\pm}, 2r+2  \right \}$. Thus, for each $r \neq \theta_{i}\in  \mathrm{{Spec}}_{L}(G)$, $\theta_{i}^{\pm}\notin S$. According to (\ref{F42}), $\theta_{i}^{\pm}\notin S$ implies that
\begin{align*}
     F_{\theta_{i}^{\pm}}(\mathbf{e}_{a}^{n+m}-\mathbf{e}_{b}^{n+m}) =&\frac{\left (\theta_{i}+2-\theta_{i}^{\pm}  \right )^2 }{\left (\theta_{i}+2-\theta_{i}^{\pm}  \right )^2+\left (2r-\theta_{i} \right )}
 \left(
	\begin{array}{c}
	\frac{1}{\theta_{i}+2-\theta_{i}^{\pm}} F_{\theta_{i}}(G)R_{G}(\mathbf{e}_{a}^{m}-\mathbf{e}_{b}^{m})\\ [0.3cm]
	   \frac{1}{\left (\theta_{i}+2-\theta_{i}^{\pm}\right )^{2}}R_{G}^{\top }F_{\theta_{i}}(G)R_{G}(\mathbf{e}_{a}^{m}-\mathbf{e}_{b}^{m})
	\end{array}
	\right)=\mathbf{0},
	\end{align*}
which is equivalent to
\begin{align}\label{GP4}
R_{G}^{\top }F_{\theta_{i}}(G)R_{G}(\mathbf{e}_{a}^{m}-\mathbf{e}_{b}^{m})=\mathbf{0}.
\end{align}
Recall that the sum of all eigenprojectors of $G$ is equal to the identity matrix. By (\ref{GP4}) and the fact \cite{Q-gr} that $R_{G}^{\top}R_{G}=A_{\mathcal{L}{ (G) }} +2I$, we have
\begin{align} \label{GP2}
R_{G}^{\top }F_{r}(G)R_{G}(\mathbf{e}_{a}^{m}-\mathbf{e}_{b}^{m})
 &=\sum_{\theta_{i}\in\mathrm{{Spec}}_{L}(G) }R_{G}^{\top }F_{\theta_{i}}(G)R_{G}(\mathbf{e}_{a}^{m}-\mathbf{e}_{b}^{m}) \nonumber\\
 &= \left(A_{\mathcal{L}{ (G) }} +2I\right)(\mathbf{e}_{a}^{m}-\mathbf{e}_{b}^{m}),
\end{align}
where $\mathcal{L}{ (G)} $ denotes the line graph of $G$ and $I$ is the identity matrix.
By (\ref{Fr}) and (\ref{GP2}), we have
\begin{align} \label{GP9}
 \left(A_{\mathcal{L}{ (G) }} +2I\right)(\mathbf{e}_{a}^{m}-\mathbf{e}_{b}^{m})
 =\frac{r(\mathbf{e}_{a}^{m}-\mathbf{e}_{b}^{m}-\mathbf{e}_{c}^{m}+\mathbf{e}_{d}^{m})}{2}.
\end{align}
Note that $r=2^t-1$ is odd. Then the nonzero entries of the right hand side of (\ref{GP9}) are fractions. But, all nonzero entries of the left hand side of (\ref{GP9}) are all integers. Therefore, (\ref{GP9}) can not occur, which leads to a contradiction.



\noindent\emph{Case 1.2.2.} All the eigenvalues in $S$ are quadratic integers. Similar to Case 1.1.2, we can get the contradiction.

\noindent\emph{Case 1.3.} $a\in V(G),b\in I(G)$.
According to Case 1.2 of Theorem \ref{hpest}, we know that if $\theta_{i}^{+}\in \Lambda^{+}_{ab,cd}$ or $\theta_{i}^{-}\in \Lambda^{+}_{ab,cd}$ is an integer, then $\theta_{i}^{\pm}$ are both integers. Similar to Case 1.1, if $\theta_i^{\pm}$ are both integers, then we have
$$
\theta_{i} =\frac{\left ( \delta-1 \right ) \left (\delta+1+r \right ) }{\delta} =\delta+2^{t}-1-\frac{2^{t}}{\delta}, ~(\delta \ge1).
$$
By (\ref{F441}), we get
\begin{align} \label{GP7}
 \theta_{i}^{+}=2^{t}+\delta,~\theta_{i}^{-}=2^{t}-\frac{2^{t}}{\delta}.
 \end{align}
Note that $\mathbf{e}_{a}^{m+n}-\mathbf{e}_{b}^{m+n}$ and $ \mathbf{e}_{c}^{m+n}-\mathbf{e}_{d}^{m+n}$ are Laplacian strongly cospectral in  $\mathcal{Q}(G)$.
According to (\ref{Eigen73}) and (\ref{Eigen72}), we have
$$
F_{r+2}\left( \mathbf e_{c}^{m+n}-\mathbf e_{d}^{m+n} \right )=\pm F_{r+2}\left( \mathbf e_{a}^{m+n}-\mathbf e_{b}^{m+n} \right ) \neq\mathbf{0},
$$
which implies that $c\in V(G), d\in I(G)$ or $c \in I(G), d\in V(G)$. Without loss of generality, we assume that $c\in V(G),~d\in I(G)$. In this case, $F_{r+2}\left( \mathbf e_{c}^{m+n}-\mathbf e_{d}^{m+n} \right )= F_{r+2}\left( \mathbf e_{a}^{m+n}-\mathbf e_{b}^{m+n} \right ) \neq\mathbf{0}$, which means that
 \begin{align} \label{GP6}
 0^{+}=r+2=2^{t}+1\in \Lambda^{+}_{ab,cd},
\end{align}
 and then all the eigenvalues in $S$ are integers. By Lemma \ref{Coutinho} (b), we have $\Delta=1$. By Lemma \ref{Coutinho} (c), for each $\gamma\in \Lambda^{+}_{ab,cd}$, $\frac{\gamma-0^+}{g}$ is even, which implies that all the eigenvalues in $\Lambda^{+}_{ab,cd}$ have the same parity, and then
(\ref{GP6}) tells us that all the eigenvalues in $\Lambda^{+}_{ab,cd}$ are odd. Thus, $2r+2\notin \Lambda^{+}_{ab,cd}$.
We consider the following cases:

If $\theta_{i}^{-}\in \Lambda^{+}_{ab,cd}$ is odd, by (\ref{GP7}), $\delta=2^t$ and
$\theta_{i}=2^{t+1}-2=2r$, which contradicts the fact that $G$ is an $r$-regular non-bipartite graph.

If $\theta_{i}^{+}\in \Lambda^{+}_{ab,cd}$ is odd, by (\ref{GP7}), $\delta=1$, $\theta_{i}=0$ and $\theta_{i}^{+}=r+2=2^{t}+1$. Thus
\begin{align} \label{GP5}
\Lambda^{+}_{ab,cd}=\left\{r+2\right\}.
\end{align}
Combining (\ref{gam+}) and (\ref{GP5}), we obtain that
\begin{align} \label{r2equa1212}
F_{r+2}(\mathbf{e}_{a}^{n+m}-\mathbf{e}_{b}^{n+m})
=\frac{\mathbf{e}_{a}^{n+m}-\mathbf{e}_{b}^{n+m}+\mathbf{e}_{c}^{n+m}-\mathbf{e}_{d}^{n+m}}{2}.
\end{align}
Left multiplying $(\mathbf{e}_{a}^{n+m}+\mathbf{e}_{c}^{n+m})^{\top}$ on both sides of (\ref{r2equa1212}), by (\ref{Eigen72}), we get that the equality of (\ref{r2equa1212}) holds if and only if $n=2$, which contradicts the fact that $r=2^{t}-1\ge3$.

\noindent\emph{Case 2.} $G$ is a bipartite graph.

\noindent\emph{Case 2.1.} $a,b\in V(G)$.

\noindent\emph{Case 2.1.1.} $a\in V_{1}(G),b\in V_{2}(G)$.
By Lemma \ref{eigenvalues} and Theorem \ref{C_n}, we know that
 $$
 S\subseteq \left \{ \theta_{i}^{\pm } \right \} \cup \left \{r\right \},~ 1\le i\le p,
 $$
 where $ \theta_1>\cdots>\theta_p=0 $ are all distinct Laplacian eigenvalues of $G$ except for $2r$. Note that $\mathbf{e}_{a}^{m+n}-\mathbf{e}_{b}^{m+n}$ and $ \mathbf{e}_{c}^{m+n}-\mathbf{e}_{d}^{m+n}$ are Laplacian strongly cospectral in  $\mathcal{Q}(G)$. According to (\ref{Eigen73}) and (\ref{F45}), we have
$$
F_{r}\left( \mathbf e_{c}^{m+n}-\mathbf e_{d}^{m+n} \right )=\pm F_{r}\left( \mathbf e_{a}^{m+n}-\mathbf e_{b}^{m+n} \right ) \neq\mathbf{0},
$$
which implies that $c\in V_{1}(G),d\in V_{2}(G)$ or $c\in V_{2}(G),d\in V_{1}(G)$. Without loss of generality, we assume that $c\in V_{1}(G),d\in V_{2}(G)$. In this case, $F_{r}(\mathbf{e}_{a}^{n+m}-\mathbf{e}_{b}^{n+m}) =F_{r}(\mathbf{e}_{c}^{n+m}-\mathbf{e}_{d}^{n+m})\neq\mathbf{0}$,
which implies that
$$
r=2^t-1\in \Lambda_{ab,cd}^{+},
$$
 and then all the eigenvalues in $S$ are integers. By Lemma \ref{Coutinho} (b) and (c) and the fact that $r\in \Lambda_{ab,cd}^{+}$ is odd, all the eigenvalues in $\Lambda_{ab,cd}^{+}$ are odd.
Recall Theorem \ref{C_n} that for each $\theta_{i}\neq2r$, $\theta_{i}^{+}\in \Lambda_{ab,cd}^{+}$ if and only if $\theta_{i}^{-}\in \Lambda_{ab,cd}^{+}$ .
Thus, for each $\theta_{i}\neq2r$, if $\theta_{i}^{\pm}\in \Lambda_{ab,cd}^{+}$, then $\theta_{i}^{\pm}$ are both odd. According to (\ref{GP7}), we know that there is no such $\delta$ satisfying that $\theta_{i}^{\pm}$ are both odd. Thus
\begin{align} \label{bi+}
\Lambda_{ab,cd}^{+}=\left \{ r\right \}.
\end{align}
By (\ref{gam+}) and (\ref{bi+}), we have
\begin{align}\label{r=newad11}
F_{r}(\mathbf{e}_{a}^{n+m}-\mathbf{e}_{b}^{n+m})
=\frac{\mathbf{e}_{a}^{n+m}-\mathbf{e}_{b}^{n+m}+\mathbf{e}_{c}^{n+m}-\mathbf{e}_{d}^{n+m}}{2}.
\end{align}
Left multiplying $(\mathbf{e}_{a}^{n+m}+\mathbf{e}_{c}^{n+m})^{\top}$ on both sides of (\ref{r=newad11}), by (\ref{F45}),  we get that the equality of  (\ref{r=newad11})  holds if and only if $n=4$. Note that $r=2^t-1 \ge3$. Thus, $G$ must be a $3$-regular complete graph, which contradicts the fact that $G$ is a bipartite graph.

\noindent\emph{Case 2.1.2.} $a,b\in V_{i}(G),~ i=1,2$. By Lemma \ref{eigenvalues} and Theorem \ref{C_n}, we know that
 $$
 S\subseteq \left \{ \theta_{i}^{\pm } \right \} ,~ 1\le i\le p,
 $$
 where $ \theta_1>\cdots>\theta_p=0 $ are all distinct Laplacian eigenvalues of $G$ except for $2r$. Next, we only need to consider two cases, that is, all the eigenvalues in $S$ are integers or quadratic integers respectively, whose proofs are similar to that of Cases 1.1.1 and 1.1.2. Hence we omit the details here.

\noindent\emph{Case 2.2.} $a,b\in I(G)$. By Lemma \ref{eigenvalues} and Theorem \ref{C_n}, we have
\begin{align*}
S\subseteq \left \{ \theta_{i}^{\pm } \right \}\cup \left \{2r+2 \right \},~ 1\le i\le p,
\end{align*}
and for each $\theta_{i}\neq2r$, $\theta_{i}^{+}\in S$ if and only if $\theta_{i}^{-}\in S$, where $ \theta_1>\cdots>\theta_p=0 $ are all distinct Laplacian eigenvalues of $G$ except for $2r$.

\noindent\emph{Case 2.2.1.}
All the eigenvalues in $S$ are integers. Similar to Case 1.2.1, we can get (\ref{nonSupport}) and (\ref{Fr}). Note that $2r\in \mathrm{{Spec}}_{L}(G)$ and
\begin{align}\label{F2rG}
  F_{2r}(G) =\frac{1}{n}
\begin{pmatrix}
 J_{\left | V_1 \right |} & -J_{\left | V_1 \right |\times\left | V_2 \right |} \\[0.2cm]
 - J_{\left | V_2 \right |\times\left | V_1 \right |} & J_{\left | V_2 \right |}
\end{pmatrix}.
	\end{align}
By (\ref{F2rG}), one can easily verify that
$$
R_{G}^{\top }F_{2r}(G)R_{G}(\mathbf{e}_{a}^{m}-\mathbf{e}_{b}^{m})=\mathbf{0}.
$$
Thus, for each $r \neq \theta_{i}\in  \mathrm{{Spec}}_{L}(G)$, we have (\ref{GP4}). Refining the proof of Case 1.2.1, we can get the  contradiction.

\noindent\emph{Case 2.2.2.}
All the eigenvalues in $S$ are quadratic integers. The proof is similar to that of Case 1.1.2, and hence we omit the details here.

\noindent\emph{Case 2.3.} $a\in V(G),b\in I(G)$. Similar to Case 1.3, in this case, we also have (\ref{GP7}), (\ref{GP6}) and all the eigenvalues in $\Lambda^{+}_{ab,cd}$ are odd with the assumption that $c\in V(G),~d\in I(G)$.  Then $2r+2\notin \Lambda^{+}_{ab,cd}$. Suppose that $\Lambda^{+}_{ab,cd}=\left\{r+2\right\}$. By (\ref{gam+}), we also have (\ref{r2equa1212}). By (\ref{Eigen72}), one can easily verify that (\ref{r2equa1212}) can not occur. Thus, there exits at least another eigenvalue ($\theta_{i}^{+}$ or $\theta_{i}^{-}$) in $\Lambda^{+}_{ab,cd}$ besides $r+2$.
Recall Case 1.3 that if $\theta_{i}^{+}\in \Lambda^{+}_{ab,cd}$ or $\theta_{i}^{-}\in \Lambda^{+}_{ab,cd}$ is an integer, then $\theta_{i}^{\pm}$ are both integers, which satisfy (\ref{GP7}). Consider the following cases:

If $\theta_{i}^{+}\in \Lambda^{+}_{ab,cd}$ is odd, by (\ref{GP7}), then $\delta=1$, $\theta_{i}^{+}=2^t+1=r+2$, which contradicts the fact that $\theta_{i}^{+}\neq r+2$.

If $\theta_{i}^{-}\in \Lambda^{+}_{ab,cd}$ is odd, by (\ref{GP7}), then $\delta=2^t$, $\theta_{i}^{-}=2^t-1=r$.
Thus,
$$
\Lambda^{+}_{ab,cd}= \left\{r+2,r\right\}.
$$
By (\ref{gam+}), we have
$$
(F_{r+2}+F_{r})(\mathbf{e}_{a}^{n+m}-\mathbf{e}_{b}^{n+m})
=\frac{\mathbf{e}_{a}^{n+m}-\mathbf{e}_{b}^{n+m}+\mathbf{e}_{c}^{n+m}-\mathbf{e}_{d}^{n+m}}{2}.
$$
By (\ref{F45}) and (\ref{Eigen72}), it is easy to verify that the above equation can not hold.

This completes the proof. \qed
\end{proof}

Let $n\ge3$ and $1\le k< \frac{n}{2} $. The \emph{generalized Petersen graph} \cite{Watkins1969}, denoted $GP(n,k)$, is the graph with vertex set
$\left\{u_j,v_j \mid  j\in \mathbb{Z}_n \right\}$ and edge set $\left\{u_ju_{j+1},v_jv_{j+k}, u_jv_j  \mid j\in \mathbb{Z}_n \right\}$, where $\mathbb{Z}_n $ is the cyclic group of order $n$. Since the generalized Petersen graph is a 3-regular graph, by Theorem \ref{2p}, we get the following result immediately.

\begin{cor}
Let $GP(n,k)$ denote the generalized Petersen graph. Then $\mathcal{Q}\left ( GP(n,k) \right )$ does not have Pair-LPST.
\end{cor}

In Theorems \ref{hpest} and \ref{2p}, we prove that the Q-graph of an $r$-regular graph does not have Pair-LPST when $r+1$ is prime or a power of $2$. If $r+1$ is neither prime nor a power of $2$, we believe that the method used in Theorems \ref{hpest} and \ref{2p} will still work, although the proof could become more complicated. In particular, for a given graph, the method will be more useful. As an example, we prove that the Q-graph of a complete graph does not have Pair-LPST in the rest of this section.


\begin{lemma}\label{eigenKn} Let $K_{n}$ be a complete graph with $n\ge4$ vertices and $m$ edges. Then
\begin{itemize}
\item[\rm (a)]  $2n$ is a Laplacian eigenvalue of $\mathcal{Q} ( K_{n}   ) $, and its corresponding eigenprojector is
\begin{align}\label{exp1}
F_{2n}=\begin{pmatrix}
\mathbf{0}&\mathbf{0} \\[0.2cm]
 \mathbf{0} & I_{m}-\left ( \frac{1}{n-2}\left ( A_{\mathcal{L}(K_{n})}+2I_{m}  \right )+\left ( \frac{2}{n(n-1)} -\frac{4}{n( n-2 ) }
 \right ) J_{m}   \right )
\end{pmatrix}.
\end{align}
\item[\rm (b)]
$0^{\pm}$ are Laplacian eigenvalues of $\mathcal{Q}   ( K_{n}  )$, and their  corresponding eigenvectors are
              \begin{align}\label{Km3434}
     F_{0^{\pm}} =&\frac{\left (2-0^{\pm}  \right )^2 }{\left (2-0^{\pm}  \right )^2+2n-2 }
 \left(
	\begin{array}{cc}
	\frac{1}{n}J_{n}& \frac{2}{n(2-0^{\pm})}J_{n\times m}\\ [0.3cm]
	\frac{2}{n(2-0^{\pm})}J_{m\times n}  &    \frac{4}{n\left (2-0^{\pm}\right )^{2}}J_{m}
	\end{array}
	\right).
	\end{align}
\item[\rm (c)]
$n^{\pm}$ are Laplacian eigenvalues of $\mathcal{Q}  ( K_{n}  )$, and their  corresponding eigenvectors are
              \begin{align}\label{k2}
     F_{n^{\pm}} =&\frac{\left (n+2-n^{\pm}  \right )^2 }{\left (n+2-n^{\pm}  \right )^2+n-2 }
 \left(
	\begin{array}{cc}
	F_{n}(K_{n}) &\frac{1}{n+2-n^{\pm}} F_{n}(K_{n})R_{K_{n}}\\ [0.3cm]
	\frac{1}{n+2-n^{\pm}} \left (  F_{n}(K_{n})R_{K_{n}} \right )^{\top }  &    \frac{1}{\left (n+2-n^{\pm}\right )^{2}}R_{K_{n}}^{\top }F_{n}(K_{n})R_{K_{n}}
	\end{array}
	\right),
	\end{align}
where
\begin{align}\label{Kn5}
F_{n}(K_{n})={I}_{n}-\frac{1}{n}J_{n}.
\end{align}
 \end{itemize}
Thus, the spectral decomposition of $L_{\mathcal{Q} \left ( K_{n} \right )}$   is written as
$$
   L_{\mathcal{Q} \left ( K_{n} \right )}=F_{0^+}+F_{0^-}+F_{n^+}+F_{n^-}+2nF_{2n}.
$$
 \end{lemma}
\begin{proof}
The Laplacian eigenvalues of $K_{n}$ are $0$ with multiplicity $1$ and $n$ with multiplicity $n-1$.
One can easily verify that the eigenprojector corresponding of $0$ is
\begin{align}\label{Kn0}
 F_{0}(K_{n})=\frac{1}{n}J_{n}.
\end{align}
Recall that the sum of all eigenprojectors of $K_{n}$ is equal to the identity matrix. Then by (\ref{Kn0}), we have (\ref{Kn5}).
Substituting (\ref{Kn5}), (\ref{Kn0}) into (\ref{F42}), we obtain (b) and (c).

By (\ref{F441}), we have
\begin{align}\label{Fp=1}
\frac{\left (\theta+2-\theta^{+}  \right )^2 }{\left (\theta+2-\theta^{+}  \right )^2+\left
(2r-\theta \right )}
+\frac{\left (\theta+2-\theta^{-}  \right )^2 }{\left (\theta+2-\theta^{-}  \right )^2+\left
(2r-\theta \right )}=1.
\end{align}
By (\ref{F42}), (\ref{Fsum}), (\ref{Fp=0}) and (\ref{Fp=1}), we have
\begin{align}\label{Fp+}
F_{\theta^{+}}+F_{\theta^{-}}=\begin{pmatrix}
 F_{\theta}&\mathbf{0} \\[0.2cm]
 \mathbf{0} &\frac{1}{2r-\theta} R_{G}^{\top}F_{\theta}R_{G}
\end{pmatrix}.
\end{align}
Note that $K_{n}$ is a $(n-1)$-regular graph. Together with (\ref{Kn0}) and (\ref{Fp+}), we have
\begin{align}\label{Kn0sum}
F_{0^{+}}+F_{0^{-}}=\begin{pmatrix}
 \frac{1}{n}J_{n}&\mathbf{0} \\[0.2cm]
 \mathbf{0} &\frac{2}{n(n-1)}J_{m}
\end{pmatrix}.
\end{align}
By (\ref{Kn5}) and (\ref{Fp+}), we have
\begin{align}\label{Kn4}
F_{n^{+}}+F_{n^{-}}=\begin{pmatrix}
 {I}_{n}-\frac{1}{n}J_{n}&\mathbf{0} \\[0.2cm]
 \mathbf{0} &\frac{1}{n-2} R_{K_{n}}^{\top }\left ( {I}_{n}-\frac{1}{n}J_{n} \right )R_{K_{n}}
\end{pmatrix}.
\end{align}
Recall that the sum of all eigenprojectors of $\mathcal{Q} (K_{n}  )$ is equal to the identity matrix. Thus
\begin{align}\label{Knsum}
F_{2n}=I-( F_{0^{+}}+ F_{0^{-}})-( F_{n^{+}}+ F_{n^{-}}).
\end{align}
By (\ref{Kn0sum}), (\ref{Kn4}), (\ref{Knsum}) and the fact \cite{Q-gr} that $R_{G}^{\top}R_{G}=A_{\mathcal{L}{ (G) }} +2I$, we have (\ref{exp1}). This completes the proof. \qed
\end{proof}

In the following, we prove that $\mathcal{Q} (K_{n} )$ does not have Pair-LPST.

\begin{example}\label{Example2}
{\em Let $K_{n}$ be a complete graph with $n\ge4$ vertices and $m$ edges. Then $\mathcal{Q}  (K_{n} )$ does not have Pair-LPST.}
\end{example}

\begin{proof}
We prove the result by contradiction.  Suppose that $\mathbf e_{a}^{m+n}-\mathbf e_{b}^{m+n} $ and $\mathbf e_{c}^{m+n}-\mathbf e_{d}^{m+n}$
 are two distinct pair states of $\mathcal{Q}  (K_{n}  )$, and $\mathcal{Q}(K_n)$  has Pair-LPST between $\mathbf e_{a}^{m+n}-\mathbf e_{b}^{m+n}$ and $\mathbf e_{c}^{m+n}-\mathbf e_{d}^{m+n}$. Let
$$
S=\mathrm{{supp}}_{L_{\mathcal{Q}(K_{n})}}\left( \mathbf e_{a}^{m+n}-\mathbf e_{b}^{m+n} \right ).
$$
By Lemma \ref{eigenKn}, we have
$$
S\subseteq \left \{0^\pm, n^{\pm}, 2n\right \}.
$$
Consider the following cases.

\noindent\emph{Case 1.} $a, b \in V(K_{n})$.
Since $K_{n}$ is a $(n-1)$-regular graph, by (\ref{Km3434}), we have
\begin{align}\label{Kn0+}
F_{0^+}(\mathbf e_{a}^{m+n}-\mathbf e_{b}^{m+n})=\mathbf{0}.
\end{align}
Then $0^{+}\notin S$. By Theorem \ref{C_n}, we know that $0^{\pm}\notin S$. By (\ref{exp1}), we have
$$
F_{2n}(\mathbf e_{a}^{m+n}-\mathbf e_{b}^{m+n})=\mathbf{0}.
$$
Then $2n\notin S$. Thus,
$$
S= \left \{n^{\pm}\right \}.
$$
 By Theorem \ref{C_n}, we have
$$
n^{\pm}\in\Lambda^{+}_{ab,cd} ~~\text{or}~~n^{\pm}\in\Lambda^{-}_{ab,cd},
$$
which means that one of $\Lambda^{\pm}_{ab,cd}$ is empty, a contradiction.

\noindent\emph{Case 2.} $a,b\in I ( K_{n} ) $.
Similar to Case 1, we also have $0^{\pm}\notin S$. Then
 $$
 S\subseteq \left \{n^{\pm},2n\right \}.
 $$
 Recall Theorem \ref{C_n} that $n^{+}\in S$ if and only if $n^{-}\in S$. Consider the following cases.

\noindent\emph{Case 2.1.} $n^{\pm}\in S$. By Theorem \ref{C_n}, we have either $n^{\pm}\in \Lambda_{ab,cd}^{+}$ or $n^{\pm}\in \Lambda_{ab,cd}^{-}$. Then by (\ref{F441}), we have
  $$
 \frac{n^{+}-{n}^{-}}{g\sqrt{\Delta}}=\frac{\sqrt{4n+1}}{g\sqrt{\Delta}},
 $$
which can not be even, since $4n+1$ is odd. This contradicts Lemma \ref{Coutinho} (c).

\noindent\emph{Case 2.2.}  $n^{\pm}\notin S$. Note that $S$ can not be empty. Thus $S=\left \{2n\right \}$. By Lemma \ref{support}, $\mathcal{Q}\left ( K_n \right ) $ does not have Pair-LPST between
 $\mathbf e_{a}^{m+n}-\mathbf e_{b}^{m+n}$ and $\mathbf e_{c}^{m+n}-\mathbf e_{d}^{m+n}$.

\noindent\emph{Case 3.} $a\in V ( K_{n} ),b\in I ( K_{n} ) $.
Note that $\mathbf e_{a}^{m+n}-\mathbf e_{b}^{m+n}$ and $\mathbf e_{c}^{m+n}-\mathbf e_{d}^{m+n}$ are Laplacian strongly cospectral.
By (\ref{Km3434}), we have
\begin{align}\label{0PlusInS11}
 F_{0^{+}}( \mathbf e_{a}^{m+n}-\mathbf e_{b}^{m+n})=&\pm F_{0^{+}}\left( \mathbf e_{c}^{m+n}-\mathbf e_{d}^{m+n} \right )\neq \mathbf{0},
\end{align}
 which implies that $c\in V(G), d\in I(G)$ or $c \in I(G), d\in V(G)$. Without loss of generality, we assume that
$c\in V ( K_{n}), d\in I ( K_{n} )$.

\noindent\emph{Case 3.1.} If $F_{n^{-}}( \mathbf e_{a}^{m+n}-\mathbf e_{b}^{m+n})=F_{n^{-}}( \mathbf e_{c}^{m+n}-\mathbf e_{d}^{m+n})$,
by (\ref{k2}), we have
\begin{align}\label{NPM-}
F_{n}(K_{n}) (\mathbf e_{a}^{n}- \mathbf e_{c}^{n})= \frac{1}{n+2-n^{-}} F_{n}(K_{n})R_{K_{n}}(\mathbf e_{b}^{m}-\mathbf e_{d}^{m}).
 \end{align}
Plugging (\ref{Kn5}) into (\ref{NPM-}), we have
\begin{align}\label{N-in+}
\mathbf e_{a}^{n}- \mathbf e_{c}^{n}= \frac{1}{n+2-n^{-}}R_{K_{n}}(\mathbf e_{b}^{m}-\mathbf e_{d}^{m}).
 \end{align}
Note that $n\ge 4$. By (\ref{F441}),
$$
n+2-n^{-}=\frac{3+\sqrt{4n+1}}{2}\ge\frac{3+\sqrt{17}}{2}.
$$
Thus (\ref{N-in+}) can not occur, a contradiction.

\noindent\emph{Case 3.2.} If $F_{n^{-}}( \mathbf e_{a}^{m+n}-\mathbf e_{b}^{m+n})=-F_{n^{-}}( \mathbf e_{c}^{m+n}-\mathbf e_{d}^{m+n})$, by (\ref{k2}) and (\ref{Kn5}), we have
\begin{align}\label{N-in-1}
n(\mathbf e_{a}^{n}+ \mathbf e_{c}^{n})-\frac{n}{n+2-n^-}R_{K_{n}}(\mathbf e_{b}^{m}+\mathbf e_{d}^{m})=\left(2-\frac{4}{n+2-n^-} \right )J_{n\times 1}.
 \end{align}
Note (\ref{0PlusInS11}) that $0^+=n+1\in S$. By Lemma \ref{Coutinho} (b),  all the eigenvalues in $S$ should be integers. Note also that $n^{-}\in S$. Otherwise, if  $n^{-}\notin S$, we have $F_{n^{-}}( \mathbf e_{a}^{m+n}-\mathbf e_{b}^{m+n})=F_{n^{-}}( \mathbf e_{c}^{m+n}-\mathbf e_{d}^{m+n})=0$. By Case 3.1, we can get the contradiction. Thus, $n^{-}$ is an integer. By (\ref{F441}), $\sqrt{4n+1}$ must be an integer. Recall that $n\ge4$. Then $n=6,12,20,30,\ldots$.

If $n=6$, by (\ref{N-in-1}), we have
$$
6\left(\mathbf e_{a}^{6}+ \mathbf e_{c}^{6}\right)-\frac{3}{2}R_{K_{6}}\left(\mathbf e_{b}^{15}+\mathbf e_{d}^{15}\right)=J_{6\times 1}.
$$
It is easy to verify that the above equation can not occur.

If $n=12,20,30,\ldots$, then there exits a vertex $v$ ($v\neq a,c$) in $K_n$ satisfying that
 $$
(\mathbf e_{v}^{n})^{\top}R_{K_{n}}(\mathbf e_{b}^{m}+\mathbf e_{d}^{m})=0.
$$
Left multiplying $(\mathbf e_{v}^{m })^{\top}$ on both sides of (\ref{N-in-1}), we have
$$
0=2-\frac{4}{n+2-n^-},
$$
which can not occur.

Thus,
$$
F_{n^{-}}( \mathbf e_{a}^{m+n}-\mathbf e_{b}^{m+n})\neq \pm \left(F_{n^{-}}( \mathbf e_{c}^{m+n}-\mathbf e_{d}^{m+n})\right),
$$
and then $\mathbf e_{a}^{m+n}-\mathbf e_{b}^{m+n}$ and $\mathbf e_{c}^{m+n}-\mathbf e_{d}^{m+n}$ are not Laplacian strongly cospectral, a contradiction.

Therefore, $\mathcal{Q}  (K_{n} )$ does not have Pair-LPST. This completes the proof. \qed
\end{proof}

\section{Pair-LPGST in Q-graph}
\label{Sec:LPST3-3}
In this section, we show that Q-graph can have Pair-LPGST. Before proceeding, we first give the following result.
\begin{lemma}\label{lemma-pgest}
Let $G$ be an $r$-regular connected graph with $n$ vertices, $m$ edges and $r\ge 2$. Let $a,b,c,d$ be four arbitrary vertices of $G$.
 \begin{itemize}
  \item[\rm (a)]  If $G$ is a non-bipartite graph, then
  \begin{equation}
  \begin{aligned}\label{PG1}
&\frac{1}{2}\left(\mathbf e_{a}^{m+n}-\mathbf e_{b}^{m+n}\right)^\top  \exp(-\mathrm{i}tL_{ \mathcal{Q} \left ( G \right )}) \left(\mathbf e_{c}^{m+n}-\mathbf
e_{d}^{m+n}\right)\\
=&\frac{1}{2}e^{-\mathrm{i}t(r+2)/2}
\sum_{\theta\in \mathrm{{Spec}}_{L}(G)} e^{-\mathrm{i}t\theta /2}\left(\mathbf e_{a}^{n}-\mathbf e_{b}^{n}\right)^\top F_{\theta}(G)\left(\mathbf e_{c}^{n}-\mathbf
e_{d}^{n}\right) \\
&\cdot\left(\cos \left( \frac{\Delta_{\theta}t}{2} \right)+\mathrm{i}\frac{\theta+2-r}{\Delta_{\theta}}\sin\left( \frac{\Delta_{\theta}t}{2}\right )\right);
  \end{aligned}
  \end{equation}
  \item[\rm (b)] If $G$ is a bipartite graph, then
\begin{equation}
  \begin{aligned}\label{PG2}
&\frac{1}{2}\left(\mathbf e_{a}^{m+n}-\mathbf e_{b}^{m+n}\right)^\top   \exp(-\mathrm{i}tL_{ \mathcal{Q} \left ( G \right )}) \left(\mathbf e_{c}^{m+n}-\mathbf
e_{d}^{m+n}\right)\\
=&\frac{1}{2}e^{-\mathrm{i}t(r+2)/2}\sum_{\theta\in \mathrm{{Spec}}_{L}(G)\setminus \left \{ 2r \right \} } e^{-\mathrm{i}t\theta /2}\left(\mathbf e_{a}^{n}-\mathbf
e_{b}^{n}\right)^\top F_{\theta}(G)\left(\mathbf e_{c}^{n}-\mathbf e_{d}^{n}\right) \\
&\cdot\left(\cos\left(\frac{\Delta_{\theta}t}{2}\right)+\mathrm{i}\frac{\theta+2-r}{\Delta_{\theta}}\sin \left(\frac{\Delta_{\theta}t}{2}\right)\right)+
\frac{1}{2}e^{-\mathrm{i}rt}\left(\mathbf e_{a}^{n}-\mathbf e_{b}^{n}\right)^\top F_{2r}(G) \left(\mathbf e_{c}^{n}-\mathbf e_{d}^{n}\right) ,
  \end{aligned}
  \end{equation}

\end{itemize}
where $\Delta_{\theta}=\sqrt{(r+2-\theta)^{2}+4\theta}$ and $F_\theta(G)$ denotes the eigenprojector corresponding to the Laplacian eigenvalue $\theta$ of $G$.
\end{lemma}
\begin{proof}
(a) Recall that
$$
\theta^{\pm}=\frac {1}{2}\left(r+2+\theta\pm  \Delta_{\theta}   \right),
$$
for each eigenvalue $\theta$ of $L_{G}$. According to Lemma \ref{eigenvalues1} and $( \ref{spect2})$, we have
 \begin{equation}
  \begin{aligned}\label{PG3}
  &\frac{1}{2}\left(\mathbf e_{a}^{m+n}-\mathbf e_{b}^{m+n}\right)^\top   \exp(-\mathrm{i}tL_{ \mathcal{Q} \left ( G \right )}) \left(\mathbf e_{c}^{m+n}-\mathbf
  e_{d}^{m+n}\right)\\
=&\frac{1}{2}e^{-\mathrm{i}t(r+2)/2}
\sum_{\theta\in \mathrm{{Spec}}_{L}(G)} e^{-\mathrm{i}t\theta /2}\left(\mathbf e_{a}^{n}-\mathbf e_{b}^{n}\right)^\top F_{\theta}(G)\left(\mathbf
e_{c}^{n}-\mathbf e_{d}^{n}\right) \\
 &\cdot\left(\sum_{\pm}e^{\mp \mathrm{i} \frac{\Delta_{\theta}t}{2}} \frac {\left(\theta+2-\theta^{\pm}\right)^{2}}{\left(\theta+2-\theta^{\pm}\right)^{2}+\left(2r-\theta \right)} \right).
  \end{aligned}
  \end{equation}
One can easily verify that
$$
\left(\theta+2-\theta^{+}\right)\left(\theta+2-\theta^{-}\right)=-2r+\theta,
$$
and
$$
\left(\left(\theta+2-\theta^{+}\right)^{2}+\left(2r-\theta \right) \right)\left(\left(\theta+2-\theta^{-}\right)^{2}+\left(2r-\theta \right)
\right)=\left(2r-\theta\right)\Delta_{\theta}^{2}.
$$
According to Euler's formula, we  simplify the inner summation in $\left(\ref{PG3} \right)$ as
\begin{equation} \label{PG4}
\sum_{\pm}e^{\mp \mathrm{i} \frac{\Delta_{\theta}t}{2}} \frac {\left(\theta+2-\theta^{\pm}\right)^{2}}{\left(\theta+2-\theta^{\pm}\right)^{2}+\left(2r-\theta
\right)}=\cos\left(\frac{\Delta_{\theta}t}{2}\right)+\mathrm{i}\frac{\theta+2-r}{\Delta_{\theta}}\sin\left(\frac{\Delta_{\theta}t}{2}\right).
\end{equation}
Combining (\ref{PG3}) and (\ref{PG4}), we get (\ref{PG1}).

(b) According to Lemma \ref{eigenvalues}, the proof of (\ref{PG2}) is similar to that of (\ref{PG1}), hence we omit the details.
\qed
\end{proof}

\begin{theorem}\label{TheRug-5.2}
Let $G$ be an $r$-regular graph with $n$ vertices and $m$ edges. Suppose that $r+1$ is prime and $G$ has Pair-LPST between $\mathbf e_{a}^{n}-\mathbf e_{b}^{n}$ and $\mathbf e_{c}^{n}-\mathbf e_{d}^{n}$ at time $\tau=\pi/g$.
\begin{itemize}
  \item[\rm (a)] If $G$ is a non-bipartite graph, then $\mathcal{Q} \left ( G \right )$  has Pair-LPGST between $\mathbf e_{a}^{m+n}-\mathbf e_{b}^{m+n}$ and $\mathbf e_{c}^{m+n}-\mathbf e_{d}^{m+n}$.
  \item[\rm (b)] If $G$ is a bipartite graph, let $V_1\cup V_2$ be the bipartition of $V(G)$.
      \begin{itemize}
      \item[\rm (b.1)] If $a, b \in V_{i},~i=1,2$, then $\mathcal{Q} \left ( G \right )$ has Pair-LPGST between $\mathbf e_{a}^{m+n}-\mathbf e_{b}^{m+n}$ and $\mathbf e_{c}^{m+n}-\mathbf e_{d}^{m+n}$.
     \item[\rm (b.2)]If $a \in V_{1}, b \in V_{2}$ and $2g\mid (r+2)$, then $\mathcal{Q} \left ( G \right )$  has Pair-LPGST between $\mathbf e_{a}^{m+n}-\mathbf e_{b}^{m+n}$ and $\mathbf e_{c}^{m+n}-\mathbf e_{d}^{m+n}$.
\end{itemize}
\end{itemize}
\end{theorem}

\begin{proof}
\noindent\emph{Case 1.} $G$ is a non-bipartite graph.
 Let $S$ be the Laplacian eigenvalue support of $\mathbf{e}_{a}^{n}-\mathbf{e}_{b}^{n}$ in $G$.
 Since $G$ has Pair-LPST between $\mathbf e_{a}^{n}-\mathbf e_{b}^{n}$ and
$\mathbf e_{c}^{n}-\mathbf e_{d}^{n}$ at time $\tau=\pi/g$, by Lemma \ref{Coutinho}, $\Delta=1$. Thus, all eigenvalues in $S$ are integers. Let $w_{\theta}$ denote the
square-free integer such that
$$
\Delta_{\theta}:=\sqrt{(r+2-\theta)^{2}+4\theta}=s_{\theta}\sqrt{w_{\theta}},
$$
for each $\theta$ in $S$ with some integer $s_{\theta}$. Note that $0\notin S$. According to Case 1.1 of Theorem \ref{hpest}, we have that
for each $\theta$ in $S$, $\Delta_{\theta}$ is irrational and $w_{\theta}>1$. By Corollary
\ref{independent}, we have
$$
\{\sqrt{w_{\theta}}:\theta\in S\}\cup\{1\},
$$
is linearly independent over $\mathbb{Q}$. By Lemma \ref{H-W}, there exist integers $l$ and $q_{\theta}$ such that
$$l\sqrt{w_{\theta}}-q_{\theta}\approx-\frac{\sqrt{w_{\theta}}}{2g}.$$
Multiplying $4s_\theta$ on both sides of the above approximation, we have
$$
(4l+\frac{2}{g} )\Delta_{\theta}\approx 4q_{\theta}s_{\theta}.
$$
Let $T=(4l+\frac{2}{g} )\pi$. Then
\begin{align}\label{PGcs}
\cos \left(\frac{\Delta_{\theta}T}{2}\right)\approx1~ \text{~~and~~} \sin \left(\frac{\Delta_{\theta}T}{2}\right)\approx0.
\end{align}
Substituting (\ref{PGcs}) into (\ref{PG1}),
\begin{align*}
&\left |\frac{1}{2}(\mathbf e_{a}^{m+n}-\mathbf e_{b}^{m+n})^\top  \exp(-\mathrm{i}TL_{ \mathcal{Q} \left ( G \right )})(\mathbf e_{c}^{m+n}-\mathbf e_{d}^{m+n}) \right |^2 \\
=&\left |\frac{1}{2}\sum_{\theta\in S}  e^{-\mathrm{i}T\theta/2} \left(\cos \frac{\Delta_{\theta}T}{2} +\mathrm{i}\frac{\theta+2-r}{\Delta_\theta} \sin \frac{\Delta_{\theta}T}{2}\right)
(\mathbf e_{a}^{n}-\mathbf e_{b}^{n})^\top F_{\theta}(G)(\mathbf e_{c}^{n}-\mathbf e_{d}^{n}) \right |^2\\
\approx&\left |\frac{1}{2}\sum_{\theta\in S}  e^{-\mathrm{i}\frac{\pi}{g} \theta }(\mathbf e_{a}^{n}-\mathbf e_{b}^{n})^\top F_{\theta}(G)(\mathbf e_{c}^{n}-\mathbf e_{d}^{n}) \right |^2\\
=&\left |\frac{1}{2}(\mathbf e_{a}^{n}-\mathbf e_{b}^{n})^\top  e^{-\mathrm{i}\frac{\pi}{g}L_{G}}(\mathbf e_{c}^{n}-\mathbf e_{d}^{n}) \right |^2.
\end{align*}
Recall that $G$ has Pair-LPST between $\mathbf e_{a}^{n}-\mathbf e_{b}^{n}$ and $\mathbf e_{c}^{n}-\mathbf e_{d}^{n}$ at time $\tau=\pi/g$. Therefore,
$$
\left|\frac{1}{2}(\mathbf e_{a}^{m+n}-\mathbf e_{b}^{m+n})^\top  \exp(-\mathrm{i}TL_{ \mathcal{Q} \left ( G \right )})(\mathbf e_{c}^{m+n}-\mathbf e_{d}^{m+n})\right|^2\approx
1,
$$
which implies that $\mathcal{Q} \left ( G \right )$  has Pair-LPGST between $\mathbf{e}_{a}^{m+n}-\mathbf{e}_{b}^{m+n}$ and $\mathbf{e}_{c}^{m+n}-\mathbf{e}_{d}^{m+n}$. This completes the proof of (a).

\noindent\emph{Case 2.} $G$ is a bipartite graph.

\noindent\emph{Case 2.1.} $a, b \in V_{i},~i=1,2$.
By (\ref{F2rG}), we have
\begin{align}\label{PGBir1}
(\mathbf e_{a}^{n}-\mathbf e_{b}^{n})^\top F_{2r}=\mathbf{0}.
\end{align}
which means that $2r\notin S$.
Let $T=(4l+\frac{2}{g} )\pi$ be as in Case 1. Substituting (\ref{PGcs}) and (\ref{PGBir1}) into (\ref{PG2}), we have
\begin{align*}
&\left |\frac{1}{2}(\mathbf e_{a}^{m+n}-\mathbf e_{b}^{m+n})^\top  \exp(-\mathrm{i}TL_{ \mathcal{Q} \left ( G \right )})(\mathbf e_{c}^{m+n}-\mathbf e_{d}^{m+n})\right |^2 \\
=&\left|\frac{1}{2}e^{-\mathrm{i}T(r+2)/2}\sum_{\theta\in \mathrm{{Spec}}_{L}(G)\setminus \left \{ 2r \right \} } e^{-\mathrm{i}T\theta /2}\left(\mathbf e_{a}^{n}-\mathbf
e_{b}^{n}\right)^\top F_{\theta}(G)\left(\mathbf e_{c}^{n}-\mathbf e_{d}^{n}\right) \right.\\&\left.
\cdot \left(\cos\left(\frac{\Delta_{\theta}T}{2}\right)+\mathrm{i}\frac{\theta+2-r}{\Delta_{\theta}}\sin \left(\frac{\Delta_{\theta}T}{2}\right)\right)\right|^{2}\nonumber\\
\approx&\left |\frac{1}{2}e^{-\mathrm{i}T(r+2)/2}\sum_{\theta\in S}  e^{-\mathrm{i}\frac{\pi}{g} \theta }(\mathbf e_{a}^{n}-\mathbf e_{b}^{n})^\top F_{\theta}(G)(\mathbf e_{c}^{n}-\mathbf e_{d}^{n}) \right |^2\\
=&\left |\frac{1}{2}\sum_{\theta\in S }  e^{-\mathrm{i}\frac{\pi}{g} \theta }(\mathbf e_{a}^{n}-\mathbf e_{b}^{n})^\top F_{\theta}(G)(\mathbf e_{c}^{n}-\mathbf e_{d}^{n}) \right |^2\\
=&\left |\frac{1}{2}(\mathbf e_{a}^{n}-\mathbf e_{b}^{n})^\top  e^{-\mathrm{i}\frac{\pi}{g}L_{G}}(\mathbf e_{c}^{n}-\mathbf e_{d}^{n}) \right |^2.
\end{align*}
Similar to Case 1, we conclude that $\mathcal{Q} \left ( G \right )$  has Pair-LPGST between $\mathbf{e}_{a}^{m+n}-\mathbf{e}_{b}^{m+n}$ and $\mathbf{e}_{c}^{m+n}-\mathbf{e}_{d}^{m+n}$, yielding (b.1).

\noindent\emph{Case 2.2.} $a \in V_{1}, b \in V_{2}$.
 According to the proof of Case 1, if $T=(4l+\frac{2}{g} )\pi$, then for each $\theta\neq 2r$ in $S$ we have (\ref{PGcs}). Recall that $2g\mid (r+2)$, we have
\begin{align}\label{PGBir}
\exp\left(\frac{-\mathrm{i}T(r+2)}{2}\right)=1.
\end{align}
Substituting (\ref{PGcs}) and (\ref{PGBir}) into (\ref{PG2}), we have
\begin{align*}
&\left |\frac{1}{2}(\mathbf e_{a}^{m+n}-\mathbf e_{b}^{m+n})^\top  \exp(-\mathrm{i}TL_{ \mathcal{Q} \left ( G \right )})(\mathbf e_{c}^{m+n}-\mathbf e_{d}^{m+n})\right |^2 \\
\approx&\left |\frac{1}{2}\sum_{\theta\in S\setminus \left \{ 2r \right \} }  e^{-\mathrm{i}\frac{\pi}{g} \theta }(\mathbf e_{a}^{n}-\mathbf e_{b}^{n})^\top F_{\theta}(G)(\mathbf e_{c}^{n}-\mathbf e_{d}^{n})+\frac{1}{2}e^{-\mathrm{i}\frac{\pi }{g}2r  }\left(\mathbf e_{a}^{n}-\mathbf e_{b}^{n}\right)^\top F_{2r}(G)\left(\mathbf e_{c}^{n}-\mathbf e_{d}^{n}\right) \right |^2\\
=&\left |\frac{1}{2}(\mathbf e_{a}^{n}-\mathbf e_{b}^{n})^\top  e^{-\mathrm{i}\frac{\pi}{g}L_{G}}(\mathbf e_{c}^{n}-\mathbf e_{d}^{n}) \right |^2.
\end{align*}
Obviously,  $\mathcal{Q} \left ( G \right )$  has Pair-LPGST between $\mathbf{e}_{a}^{m+n}-\mathbf{e}_{b}^{m+n}$ and $\mathbf{e}_{c}^{m+n}-\mathbf{e}_{d}^{m+n}$ if $G$ has
Pair-LPST from $\mathbf e_{a}^{n}-\mathbf e_{b}^{n}$ to $\mathbf e_{c}^{n}-\mathbf e_{d}^{n}$ at time $\tau=\pi/g$, leading to (b.2).
\qed
\end{proof}

\begin{example}\label{Example-4-1}
{\em Let $C_{4}$ be the cycle with $4$ vertices. Set $V(C_{4})=\left\{v_j \mid  j\in \mathbb{Z}_4 \right\}$ and $E(C_{4})=\left\{v_jv_{j+1}\mid  j\in \mathbb{Z}_4 \right\}$. It has been proved \cite[Theorem 7.3]{ChG19} that $C_{4}$ has Pair-LPST between $\mathbf e^{4}_{v_0}-\mathbf e^{4}_{v_1}$ and $\mathbf e^{4}_{v_2}-\mathbf e^{4}_{v_3}$ at time $\tau=\pi/2$. By Lemma \ref{Coutinho}, we know that $g=2$. Notice that $C_{4}$ is a $2$-regular graph, i.e., $r=2$. Then $r+1=2+1=3$ is prime and $2g \mid (r+2)$. By Theorem \ref{TheRug-5.2} (b.2), $\mathcal{Q} \left (  C_{4} \right )$ has Pair-LPGST between
$\mathbf e^{8}_{v_0}-\mathbf e^{8}_{v_1}$ and $\mathbf e^{8}_{v_2}-\mathbf e^{8}_{v_3}$. \qed
}
\end{example}


\begin{thebibliography}{99}
\small{


\bibitem{Ack}
E. Ackelsberg, Z. Brehm, A. Chan, J. Mundinger,  C. Tamon, Laplacian state transfer in coronas, Linear Algebra Appl. 506 (2016) 154--167.

\bibitem{AckBCMT16}
E. Ackelsberg, Z. Brehm, A. Chan, J. Mundinger,  C. Tamon, Quantum state transfer in coronas, Electron. J. Combin. 24 (2) (2017) \#P2.24.

\bibitem{Alvi}
 R. Alvir, S. Dever, B. Lovitz, J. Myer, C. Tamon, Y. Xu, H. Zhan, Perfect state transfer in Laplacian quantum walk, J. Algebraic Combin. 43 (4) (2016) 801--826.

\bibitem{BOSE2}
S. Bose, Quantum communication through an unmodulated spin chan, Phys. Rev. Lett. 91 (20) (2003) 207901.

\bibitem{BOSE1}
S. Bose, A. Casaccino, S. Mancini, S. Serverini, Communication in XYZ all-to-all quantum networks with a missinng link, Int. J. Quantum Inf. 7 (4) (2009) 713--723

\bibitem{CAO}
X. Cao, Perfect edge state transfer on cubelike graphs, Quantum Inf. Process. 20 (9) (2021) 285.

\bibitem{CaoCL20}
X. Cao, B. Chen, S. Ling, Perfect state transfer on Cayley graphs over dihedral groups: the non-normal case,  Electron. J. Combin. 27 (2) (2020) \#P2.28.

\bibitem{CaoF21}
X. Cao, K. Feng, Perfect state transfer on Cayley graphs over dihedral groups, Linear Multilinear Algebra  69 (2) (2021) 343--360.

 \bibitem{CAO2}
X. Cao, J. Wan, Perfect edge state transfer on abelian Cayley graphs, Linear Algebra Appl. 653 (2022) 44--65.

\bibitem{CaoWF20}
X. Cao, D. Wang, K. Feng, Pretty good state transfer on Cayley graphs over dihedral groups, Discrete Math. 343 (1)  (2020) 111636.

\bibitem{QCh18}
Q. Chen, Edge state transfer, Master thesis, University of Waterloo, 2018.

\bibitem{chris2}
M. Christandl, N. Datta, T. Dorlas, A. Ekert, A. Kay, A. Landahl, Perfect transfer of arbitrary states in quantum spin networks, Phys. Rev. A 71 (2005) 032312.

\bibitem{chris1}
M. Christandl, N. Datta, A. Ekert, A. Landahl, Perfect state transfer in quantum spin networks, Phys. Rev. Lett. 92 (18) (2004) 187902.

\bibitem{ChG19}
Q. Chen, C. Godsil, Pair state transfer, Quantum Inf. Process. 19 (9) (2020) 321.

\bibitem{Coutinho14}
G. Coutinho, Quantum State Transfer in Graphs, PhD thesis, University of Waterloo, 2014.

\bibitem{Cohgo16}
G. Coutinho, C. Godsil, Perfect state transfer in products and covers of graphs, Linear Multilinear Algebra 64 (2) (2016) 235--246.

\bibitem{Coutinho11}
G. Coutinho, H. Liu, No Laplacian perfect state transfer in trees, SIAM J. Discrete Math. 29 (4) (2015) 2179--2188.

\bibitem{Q-gr}
D. Cvetkovi{\'c}, P. Rowlinson,  S. Simi{\'c}, An Introduction to the Theory of Graph Spectra, First Edition, Cambridge University Press, New York, 2010.

\bibitem{CGodsil}
C. Godsil, State transfer on graphs, Discrete Math. 312 (1) (2012) 129--147.

\bibitem{Godsil12}
C. Godsil, When can perfect state transfer occur? Electron. J. Linear Algebra 23 (2012) 877--890.
\bibitem{Hw}
G. H. Hardy, E. M. Wright, An Introduction to the Theory of Numbers, Fifth Edition, Oxford University Press, New York, 2000.

\bibitem{Kirk}
S. Kirkland, S. Severini, Spin-system dynamics and fault detection in threshold networks, Phys. Rev. A 83 (1) (2011) 012310.

\bibitem{LiY2}
 Y. Li, X. Liu, S. Zhang, Laplacian state transfer in Q-graph, Appl. Math. Comput. 384 (2020) 125370.

\bibitem{LiY1}
 Y. Li, X. Liu, S. Zhang, Laplacian state transfer in edge coronas, Linear Multilinear Algebra  70 (6) (2022) 24, 1023--1046.

\bibitem{LiLZZ21}
Y. Li, X. Liu, S. Zhang, S. Zhou, Perfect state transfer in NEPS of complete graphs, Discrete Appl. Math. 289 (2021) 98--114.

\bibitem{LiuD}
D. Liu, X. Liu, No Laplacian perfect state transfer in total graphs, Discrete Math. 346 (9) (2023) 113529.

\bibitem{LCXC21}
G. Luo, X. Cao, G. Xu, Y. Cheng, Cayley graphs of dihedral groups having perfect edge state transfer, Linear Multilinear Algebra  70 (20) (2022) 5957--5972.

\bibitem{MAALA}
C. D. Meyer, Matrix Analysis and Applied Linear Algebra, Siam, Philadelphia,  2000.

\bibitem{Ri}
I. Richards, An application of Galois theory to elementary arithmetic, Adv. Math. 13 (1974) 268--273.

\bibitem{Tan19}
Y. Tan, K. Feng, X. Cao, Perfect state transfer on abelian Cayley graphs, Linear Algebra Appl. 563 (2019) 331--352.

\bibitem{Tm19}
I. Thongsomnuk, Y. Meemark, Perfect state transfer in unitary Cayley graphs and gcd-graphs, Linear Multilinear Algebra  67 (1) (2019) 39--50.



\bibitem{WJ}
J. Wang, X. Liu, Laplacian state transfer in edge complemented coronas, Discrete Appl. Math. 293 (2021) 1--14.

\bibitem{liu2023}
W. Wang, X. Liu, J. Wang, Laplacian pair state transfer in vertex coronas, Linear Multilinear Algebra 71 (14) (2023) 2282--2297.

\bibitem{Watkins1969}
M. E. Watkins, A theorem on tait colorings with an application to the generalized Petersen graphs, J. Combin. Theory 6 (1969) 152--164.

\bibitem{SZ}
S. Zheng, X. Liu, S. Zhang, Perfect state transfer in NEPS of some graphs, Linear Multilinear Algebra 68 (8) (2020) 1518--1533.

}
\end{thebibliography}
\end{document}